%% file: Main.tex
\documentclass[letterpaper]{scrartcl}

\usepackage{mathtools, amsfonts, amssymb, amsthm, graphicx, hyperref,float,subcaption,fancyhdr,titling,tensor,mathrsfs,tikz,tkz-euclide,units,authblk}
\usepackage[inline]{enumitem}
\usetikzlibrary{decorations.pathmorphing}

\usepackage[margin=.7in]{geometry}

\newtheorem{theorem}{Theorem}[section]
\newtheorem{lemma}[theorem]{Lemma}

\newtheorem{proposition}[theorem]{Proposition}

\theoremstyle{definition}
\newtheorem{definition}[theorem]{Definition}

\theoremstyle{remark}
\newtheorem*{remark}{Remark}

\makeatletter
\newcommand{\problem}{\@ifstar{\problemStar}{\problemNoStar}}
\newcommand{\problemStar}[1][]{\vspace{2\baselineskip}\refstepcounter{problem}{\noindent\large \bfseries Problem~#1}\\}
\newcommand{\problemNoStar}{\vspace{2\baselineskip}\refstepcounter{problem}{\noindent\large \bfseries Problem~\arabic{set}-\arabic{problem}}\\}
\makeatother

\makeatletter
\renewcommand*\env@matrix[1][*\c@MaxMatrixCols c]{%
	\hskip -\arraycolsep
	\let\@ifnextchar\new@ifnextchar
	\array{#1}}
\makeatother

\title{Scaling Limits of Fluctuations of Extended-Source Internal DLA} 
\author{David Darrow\thanks{Supported in part by NSF Grant DMS 1500771.}} 
\date{\today}

\newcommand{\ql}{\textquotedblleft}
\newcommand{\qr}{\textquotedblright{}}
\newcommand{\pdrv}[2]{\frac{\partial #1}{\partial #2}}

\newcommand{\eps}{\varepsilon}
\newcommand{\Eps}{\mathcal{E}}
\newcommand{\ol}[1]{{\overline{#1}}}

\newcommand{\upto}{\nearrow}

\newcommand{\op}{\operatorname}
\newcommand{\supp}{\operatorname{supp}}

\begin{document}
\maketitle
\thispagestyle{fancy}
\begin{abstract}
	\noindent\textbf{Abstract.} In a previous work, we showed that the 2D, extended-source internal DLA (IDLA) of Levine and Peres is $\delta^{3/5}$-close to its scaling limit, if $\delta$ is the lattice size. In this paper, we investigate the scaling limits of the fluctuations themselves. Namely, we show that two naturally defined error functions, which measure the \ql lateness\qr{} of lattice points at one time and at all times, respectively, converge to geometry-dependent Gaussian random fields. We use these results to calculate point-correlation functions associated with the fluctuations of the flow. Along the way, we demonstrate similar $\delta^{3/5}$ bounds on the fluctuations of the related \emph{divisible sandpile} model of Levine and Peres.
\end{abstract}
\tableofcontents

\section{Introduction}
Internal diffusion-limited aggregation (IDLA) is a lattice growth model, tracking the growth of a random set $A(t)\subset\mathbb{Z}^d$ defined as follows. At each time $t$, we start a particle at the origin, and we let it undergo a simple random until it first exits the set $A(t-1)$---supposing it exits at the point $z_t$, we set $A(t)=A(t-1)\cup\{z_t\}$. Intuitively, this process follows the diffusion of particles from an origin-centered source. In fact, it was originally proposed by the chemical physicists Meakin and Deutch \cite{doi:10.1063/1.451129} in order to model such diffusive processes, such as the smoothing of a spherical surface by electrochemical polishing.

We are interested in a generalization of this model to the extended-source case, wherein particles start instead from discretizations of a fixed mass distribution, and the lattice size is allowed to grow arbitrarily small. This generalization was first introduced and studied by Levine and Peres \cite{Levine_2008}, although it corresponds to Diaconis and Fulton's earlier notion of a \ql smash sum\qr{} of two sets \cite{stanford1991growth}.

In both cases, a primary question of study is the overall smoothness of the occupied set $A(t)$. Following the work of Lawler, Bramson, and Griffeath \cite{lawler1992}, it is well-known that---in the point-source case---these sets closely approximate an origin-centered ball for large $t$. Several authors have shown strong convergence rates for this process \cite{lawler1995,asselah:hal-00795850}; most recently, Jerison, Levine, and Sheffield proved that the fluctuations away from the disk are at most of order $\log t$ in dimension 2, narrowly improving a $\log^2t$ result by Asselah and Gaudillière \cite{10.2307/23072157, asselah2013}. Independent works by Asselah and Gaudillière and by Jerison, Levine, and Sheffield proved bounds of order $\sqrt{\log t}$ in higher dimensions \cite{Asselah_2013,jerison2013}, which have been shown to be tight \cite{asselah2011lower}. In the extended-source case, Levine and Peres first showed that the scaling limits of IDLA correspond to solutions of a closely related free boundary problem \cite{Levine_2008}. We recently proved that, if the lattice size is $\delta$, the fluctuations of IDLA away from this expected set are at most of order $\delta^{3/5}$.

The fluctuations can also be studied \ql on the aggregate\qr{}, however, which provides interesting insight into the geometry of the problem. Namely, we are interested in studying mean fluctuations over an area of finite volume, as weighted by a test function $u\in C^\infty(\mathbb{R}^d)$. To do this in the point-source case, Jerison et al \cite{jerison2014} introduced natural error functions on the lattice $\delta\mathbb{Z}^d$, which quantify how late or early the IDLA process is in getting to a given point. Specifically, they introduced a \emph{fluctuation function} $E^s$ and a \emph{lateness function} $L$, that capture fluctuations at a single time $s$ and at all times, respectively. They proved that these error functions weakly approach certain Gaussian random fields as the lattice spacing $\delta$ decreases, allowing them to find the scaling limits of fluctuations integrated against a test function $u$. Eli Sadovnik studied this question more recently for an extended source, in the special case of the single-time fluctuation function and with discrete harmonic test functions \cite{eli16}.

In this paper, we extend the techniques used in \cite{jerison2014} and \cite{eli16} in order to prove more general scaling limits of random error functions in the extended-source case. Our main results, Theorems \ref{fixedtime} and \ref{orderedtime}, show that the fluctuation function $E^s$ and the lateness function $L$ converge weakly to geometry-dependent Gaussian random fields, allowing for any $C^4$ test functions. In particular, by choosing highly localized test functions, we will be able to calculate \ql point-correlation functions\qr{}, which encode the correlations between fluctuations of IDLA at two different points in space.

It must be noted that, in the point-source case of \cite{jerison2014}, the functions $E^s$ and $L$ measure fluctuations away from a previously-calculated continuous limit of IDLA; specifically, they measure the difference between $A(t)$ and a smooth sphere. To our knowledge, this is not possible in the general-source case without stronger estimates on the convergence of discrete harmonic functions---as such, our general-source versions of $E^s$ and $L$ compare IDLA to a closely related deterministic process: the \emph{discrete sandpile} model of Levine and Peres \cite{Levine_2010}. We show in Theorem \ref{narrowsand} that the discrete sandpile converges at least as quickly as the best known estimates (from \cite{darrow2020convergence}) on IDLA. In fact, we believe that it converges faster than IDLA, but the estimate from Theorem \ref{narrowsand} is sufficient for our purposes.

After briefly reviewing the necessary theory, we introduce our primary results in Section \ref{mainresults}. The following sections are spent proving these results; Section \ref{fixedtimeproof} proves the scaling limit of the fixed-time fluctuation function, and Section \ref{orderedtimeproof} proves that of the lateness function. Finally, we use these results to calculate point correlation functions of IDLA fluctuations in Section \ref{correlations}.

\input{0_Review}
\input{1_MainResult}
\input{2_FixedTime}
\input{3_OrderedTime}
\input{4_Correlations}
\input{5_Conclusion}

\input{_Appendix}

\subsection*{Acknowledgments}
I would like to thank Professor David Jerison and Pu Yu (MIT Department of Mathematics) for their mentorship throughout this project, as well as Professor Scott Sheffield for his insights regarding the divisible sandpile model.

\bibliographystyle{alpha}
\bibliography{../Bibliography}

\end{document}

%% file: 0_Review.tex
\section{Review of lattice growth processes}
Here we provide a background on extended-source IDLA and on a related deterministic process, the \emph{divisible sandpile growth} introduced also by Levine and Peres \cite{Levine_2008}. Many of our specific definitions are taken from our preceding paper, \cite{darrow2020convergence}; see that paper for more information.

Following from \cite{darrow2020convergence}, we restrict attention to IDLA processes started from \emph{concentrated mass distributions}. 

\begin{definition}
	Let $D_0\subset\mathbb{R}^2$ be a compact, connected domain with smooth boundary, and fix $N\in\mathbb{Z}^{\geq 0}$ and $T_1,...,T_N\in\mathbb{R}^{\geq 0}$. For each $i=1,...,N$ and $s\in[0,T_i]$, suppose $Q^s_i\subset D_0$ satisfies the following properties:
	\begin{enumerate}
		\item $Q_i^s$ is a compact domain with $\op{Vol}(Q_i^s)=s$.
		\item $Q_i^s$ is bounded away from $\partial D_0$---that is, $Q_i^s\subset\subset\op{int}(D_0)$.
		\item $Q_i^{s}\subset Q_i^{s'}$ for $s\leq s'\leq T_i$.
		\item $\partial Q_i^s$ is rectifiable, with arclength bounded independently of $s$.
	\end{enumerate}
	Finally, set $T=\sum_kT_k$, and fix increasing functions $s_i:[0,T]\to[0,T_i]$ satisfying $\sum_ks_k(s)=s$ for all $s\in[0,T]$.
	In this setting, the \emph{concentrated mass distribution} associated to the data $(D_0,\{Q^s_i\},\{s_i\})$ is the map $\sigma_s:\mathbb{R}^2\to\mathbb{Z}^{\geq 0}$ defined by
	\[\sigma_s=\mathbf{1}_{D_0}+\sum_{i=1}^{N}\mathbf{1}_{Q^{s_i(s)}_i}.\]
\end{definition}
In short, a concentrated mass distribution is a collection of increasing subsets $Q^{s_i}_i$ of $D_0$, such that the total mass at any time $s$ is $\op{vol}(D_0)+s$. The functions $s_i$ give the mass of each subset $Q^{s_i}_i$ at the time $s$.

The analysis of this paper holds in its entirety for infinite mass distributions, where $T=\infty$. For these, we simply require that the finite-time collections $\{Q_i^s\}|_{s\leq T'}$ and $\{\sigma_s\}|_{s\leq T'}$ give rise to concentrated mass distributions for any $T'>0$. We will assume that $T<\infty$, but we can also imagine that we have simply ``cut off'' an infinite mass distribution in the manner just described.

Since we are studying processes on discrete lattices, we are primarily interested in the restrictions of these mass functions to the grid $\frac{1}{m}\mathbb{Z}^2$. Write $S^m_s$ for the multiset defined by $\sum\nolimits_{\frac{1}{m}\mathbb{Z}^2}(\sigma_s-\mathbf{1}_{D_0})$; that is, for any $z\in\frac{1}{m}\mathbb{Z}^2\cap\bigcup Q_i^{s_i}$, we have that $z\in S^m_s$ with multiplicity $\sigma_s(z)-1$.
We can order $S^m_T$ into a sequence $z_{m,j}$ of source points as follows:
\begin{enumerate}
	\item If $z\in S^m_T$ with multiplicity $k$, let $s_i(z):=\inf\{t\;|\;z\in S^m_t\;\text{with multiplicity}\;i+1\}$ for $i\leq k-1$.
	\item Given $\{z_{m,1},...,z_{m,j-1}\}$, choose $z_{m,j}\in S^m_T\setminus \{z_{m,1},...,z_{m,j-1}\}$ to minimize $s_{i(z,j)}(z)$, where $i(z,j)$ is the multiplicity of $z$ in $\{z_{m,1},...,z_{m,j-1}\}$.
\end{enumerate}
In short, we are simply ordering the particles in $S^m_T$ in the order they appear (accounting for multiplicity) in the sets $\{S^m_s\}$. Given this sequence, we define the discrete densities $\sigma_{m,n}=\mathbf{1}_{\frac{1}{m}\mathbb{Z}^2\cap D_0}+\sum_{i=1}^n\mathbf{1}_{\{z_{m,i}\}}$. It is clear that $\sigma_{m,n}$ differs from its continuum limit $\sigma_{n/m^2}$ at only $O(m^{-2})$ points, accounting for multiplicity.

The (resolution $m$) \emph{internal DLA (IDLA)} associated with the mass distribution is the following process:
\begin{definition}[Internal DLA]
	 Suppose we have a concentrated mass distribution with initial set $D_0$ giving rise to the sequences $\{z_{m,t}\}$. The IDLA $A_m(t)$ associated with the mass distribution is as follows. Define the initial set $A_m(0)=\frac{1}{m}\mathbb{Z}^2\cap D_0$. Then, for each integer $t\geq 1$, start a simple random walk at $z_{m,t}$, and let $z'_t$ be the first point in the walk outside the set $A_m(t-1)$---then $A_m(t):=A_m(t-1)\cup\{z'_t\}$.
	 
	 Importantly, the law of $A_m(t)$ does not depend on the order of $\{z_{m,1},...,z_{m,t}\}$, as proven by Diaconis and Fulton \cite[Lemma 2.2]{stanford1991growth} (see \cite[Section 3]{stanford1991growth} for the application of this result to our setting).
\end{definition}
For each time $s$, the sets $A_m(m^2s)$ approach a deterministic limit $D_s$ almost surely, where $D_s$ is the Diaconis--Fulton ``smash'' sum
\[D_s=D_0\oplus Q_1^{s_1}\oplus\cdots\oplus Q_N^{s_N}.\]
The smash sum operation is as defined in \cite{Levine_2010}:
\begin{definition}
	If $A,B\subset\frac{1}{m}\mathbb{Z}^2$, we define the discrete smash sum $A\oplus B$ as follows. Let $C_0=A\cup B$, and for each $x_i\in \{x_1,...,x_n\}=A\cap B$, start a simple random walk at $x_i$ and stop it upon exiting $C_{i-1}$. Let $y_i$ be its final position, and define $C_{i}=C_{i-1}\cup\{y_i\}$. Then $A\oplus B:=C_n$ is a random set.
	
	As proven in \cite[Theorem 1.3]{Levine_2010}, if we instead take domains $A,B\subset\mathbb{R}^2$, the smash sums $A_m\oplus B_m$ of
	\[A_m:=\frac{1}{m}\mathbb{Z}^2\cap A,\qquad B_m:=\frac{1}{m}\mathbb{Z}^2\cap B\]
	approach a deterministic limit, which we label $A\oplus B$. Figure \ref{squarefig} (taken from \cite{darrow2020convergence}) gives an example of this.
\end{definition}
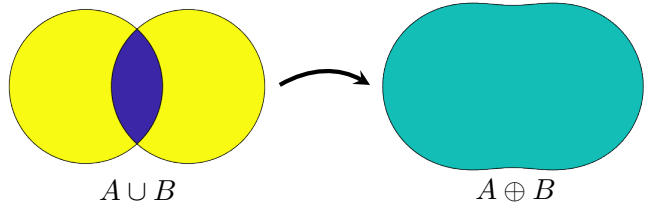
\begin{figure}[H]
	\centering
	\begin{tikzpicture}
		\node[anchor=south west,inner sep=0] (image) at (0,0) {\includegraphics[scale=.3]{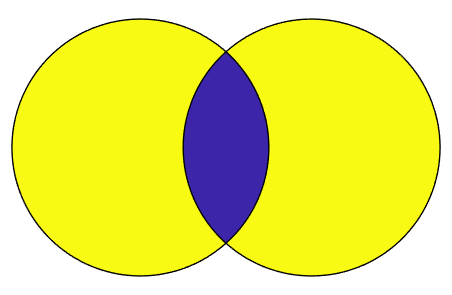}};
		\node[text width=2cm,rotate=0] at (2.3,-.2) {$A\cup B$};
		{}
		\path[ultra thick,->,>=stealth] (3.7,1.2) edge[bend left] (4.9,1.2);
		\node[anchor=south west,inner sep=0] (image) at (5,0) {\includegraphics[scale=.3]{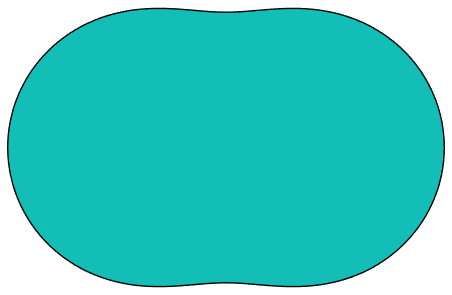}};
		\node[text width=2cm,rotate=0] at (7.3,-.2) {$A\oplus B$};
		{}
	\end{tikzpicture}%
	\caption{The smash sum $A\oplus B$ is the deterministic limit of an IDLA-type growth process starting from the sets $A$ and $B$, representing the dispersal of particles in $A\cap B$ (in dark blue above) to the edges of $A\cup B$ (in yellow above). }\label{squarefig}
\end{figure}
The convergence $A_m(m^2s)\to D_s$ was shown originally by Levine and Peres \cite{Levine_2010}. In \cite[Theorem 3.1]{darrow2020convergence}, we have recently shown the following convergence rate for this scaling limit, in the special case that $D_s$ is a \emph{smooth} flow---that is, that $s\mapsto D_s$ is a smooth isotopy for $s\in[0,T]$.
\begin{lemma}\label{narrow}
	Suppose $D_s$ is a smooth flow arising from a concentrated mass distribution. For large enough $m$, the fluctuation of the associated IDLA $A_m(t)$ is bounded as
	\[
	\mathbb{P}\bigg\{(D_{s})_{Cm^{-3/5}}\cap\frac{1}{m}\mathbb{Z}^2\subset A_m(m^2s)\subset (D_{s})^{Cm^{-3/5}}\;\text{for all}\;s\in[0,T]\bigg\}^c\leq e^{-m^{2/5}}\]
	for a constant $C$ depending on the flow, where $(D_s)^\eps$ and $(D_s)_\eps$ denote outer- and inner-$\eps$-neighborhoods of $D_s$, respectively.
\end{lemma}
In other words, the fluctuations of the random set $A_m(m^2s)$ are unlikely to be larger in magnitude than $Cm^{-3/5}$, for some fixed $C>0$. We will also use this to bound the maximum fluctuations of a closely related process---the \emph{divisible sandpile growth}---defined as follows:
\begin{definition}[Divisible Sandpile]
	Suppose we have a concentrated mass distribution with initial set $D_0$ giving rise to the sequences $\{z_{m,i}\}$. The \emph{divisible sandpile aggregation} associated with our mass distribution is characterized by its \emph{final mass distributions} $\nu_{m,t}$. Let $\nu_{m,0}:=\mathbf{1}_{A_m(0)}=\mathbf{1}_{D_0\cap \frac{1}{m}\mathbb{Z}^2}$, and define $\nu_{m,n}$ inductively as follows. 
	
	Given $\nu_{m,n}$, define the intermediate function $\nu_{m,n}^0=\nu_{m,n}+\mathbf{1}_{\{z_{m,n+1}\}}$. At each time step $t$, choose a point $z=z(t)\in\supp\nu_{m,n}^t$ such that $\nu_{m,n}^t(z)>1$. Set
	\[\nu_{m,n}^{t+1}(z)=1,\qquad\nu_{m,n}^{t+1}(z\pm 1/m)=\nu_{m,n}^{t}(z\pm 1/m)+\frac{1}{4}(\nu_{m,n}^{t}(z)-1),\]
	\[\nu_{m,n}^{t+1}(z\pm i/m)=\nu_{m,n}^{t}(z\pm i/m)+\frac{1}{4}(\nu_{m,n}^{t}(z)-1).\]
	In other words, we define $\nu_{m,n}^{t+1}$ by taking the ``excess mass'' at $z$ in $\nu_{m,n}^t$ and splitting it evenly between the neighbors of $z$. For a large enough (but finite) $t'$, we will have $\nu_{m,n}^{t'}\leq 1$ everywhere, and the above process must stop; then define $\nu_{m,n+1}=\nu_{m,n}^{t'}$.
\end{definition}
The following proposition follows directly from the definition of the divisible sandpile:
\begin{proposition}
	Suppose $h:\frac{1}{m}\mathbb{Z}^2\to\mathbb{R}$ is discrete harmonic on $\supp\nu_{m,n}$. Then,
	\[\sum_{z\in\frac{1}{m}\mathbb{Z}^2} h(z)\nu_{m,n}(z)=\sum_{z\in\frac{1}{m}\mathbb{Z}^2\cap D_0}h(z)\sigma_{m,n}(z).\]
\end{proposition}
Finally, we find the same $m^{-3/5}$-bound on maximum fluctuations for the divisible sandpile as we do for IDLA; the following theorem is proved in the \hyperref[appn]{Appendix}:
\begin{theorem}\label{narrowsand}
	Suppose $D_\tau$ is a smooth flow arising from a concentrated mass distribution. For large enough $m$ and any time $s\in[0,T]$, the fluctuations of the occupied set $\supp\nu_{m,m^2s}$ are bounded as
	\[(D_{s})_{Cm^{-3/5}}\cap\frac{1}{m}\mathbb{Z}^2\subset \supp\nu_{m,m^2s}\subset (D_{s})^{Cm^{-3/5}}\]
	for a constant $C$ depending on the flow.
\end{theorem}
To simplify notation, we will continue to write $\eps_m=Cm^{-3/5}$, as in Lemmas \ref{narrow} and \ref{narrowsand}. We further define
\[F^s_m=(D_s)^{\eps_m}\setminus(D_s)_{\eps_m}.\]

%% file: 1_MainResult.tex
\section{Main results}\label{mainresults}
Our two primary results pertain to scaling limits of the fluctuations of $A_m(t)$ away from its deterministic limit. Following \cite{jerison2014}, we quantify these fluctuations using the following random functions.

First, the \emph{(time $s$) error function} $E^s_m:\frac{1}{m}\mathbb{Z}^2\to\mathbb{R}$ is defined as
\[E^s_m(x):=m\left(1_{A_{m}(m^2s)}(x)-\nu_{m,m^2s}(x)\right).\]
This takes a positive value on ``early'' points, where the IDLA $A_m$ has reached by time $m^2s$ but where the expected set---represented here by the divisible sandpile occupied set---has not yet reached. It takes a negative value on ``late'' points, where the divisible sandpile occupied set has reached but the IDLA has not.

Although $E^s_m$ itself does not converge (in $m$) to a well-defined random variable, our primary objects of interest are the limits of inner products 
\[(E^s_m,u)=m^{-2}\sum_{x\in\frac{1}{m}\mathbb{Z}^2}E^s_m(x)u(x)=m^{-1}\sum_{x\in\frac{1}{m}\mathbb{Z}^d}u(x)\left(1_{A_{m}(m^2s)}(x)-\nu_{m,m^2s}(x)\right).\]
We can think of $(E^s_m,u)$ as a snapshot of the discrepancy at the fixed time $s$, weighted by the function $u\in C^4(\mathbb{R}^2)$.

Through the following theorem, we show that $E^s_m$ converges weakly to a Gaussian random field on the fixed-time curve $\partial D_s$:
\begin{theorem}\label{fixedtime}
	Suppose $u\in C^4(\mathbb{R}^d)$. The random variables $(E^s_m,u)$ converge in law to a normal variable of mean 0 and variance
	\[\int_{D_s}|\psi|^2(1-\sigma_s),\]
	where $\psi$ solves the Laplace problem on $D_s$ with boundary values $\psi|_{\partial D_s}\equiv u|_{\partial D_s}$.
\end{theorem}
Of course, we can turn this into a covariance formula using a polarization identity; if $u,v\in C^4(\mathbb{R}^2)$, the variables $(E^s_m,u)$ and $(E^s_m,v)$ form a joint Gaussian random variable with covariance
\begin{equation}\label{covariance1}
	\int_{D_s}\psi\varphi(1-\sigma_s),
\end{equation}
where $\psi$ and $\varphi$ are harmonic on $D_s$ with $\psi|\partial D_s\equiv u$ and $\varphi|\partial D_s\equiv v$, respectively.

A more sensitive metric is given by the \emph{lateness function},
\[L^s_m=\sum_{n=1}^{\lfloor m^2s\rfloor}\frac{n}{m}1_{A_{m,n}\setminus A_{m,n-1}}-\sum_{n=1}^{\lfloor m^2s\rfloor}\frac{n}{m}\left(\nu_{m,n}-\nu_{m,n-1}\right).\]
Up to a scaling factor, the first term in this expression is the actual time of arrival at each point. The latter term is an approximation of the expected time of arrival, which we can see as follows.

Suppose $x\in\frac{1}{m}\mathbb{Z}^d$, and $\langle T_{m,x}\rangle$ is the expected time for $A_{m,n}$ to arrive at $x$. For a brief window around $T_x$, the quantity $\nu_{m,n}(x)$ lies strictly between $0$ and $1$---say, when $n\in\{n',n'+1,...,n'+\Delta-1\}$. As $\nu_{m,n}(x)$ is constant before $n'$ and after $n'+\Delta$, only the terms involving $\{n',n'+1,...,n'+\Delta-1\}$ contribute to the sum 
\begin{equation}\label{sumthing}
	\sum_{n=1}^{m^2s}n\left(\nu_{m,n}-\nu_{m,n-1}\right)(x).
\end{equation}
Of course, the increments $\left(\nu_{m,n}-\nu_{m,n-1}\right)(x)$ are non-negative, and 
\[\sum_{n=n'}^{n'+\Delta}\left(\nu_{m,n}-\nu_{m,n-1}\right)(x)=\nu_{m,n'+\Delta}(x)-\nu_{m,n'-1}(x)=1,\]
so the sum (\ref{sumthing}) is a weighted average of $\{n',n'+1,...,n'+\Delta\}$. As this interval is tightly centered around $\langle T_{m,x}\rangle$, we expect the overall sum to converge (in $m$) to $\langle T_{m,x}\rangle$. 

Our second result is in the same spirit as Theorem \ref{fixedtime}, showing now that the lateness function converges weakly to a 2D Gaussian random field:
\begin{theorem}\label{orderedtime}
	Suppose $u\in C_0^4(\mathbb{R}^d)$, with $\supp u\subset D_s$ (for instance, if $s\to\infty$). The random variables $(L^s_m,u)$ converge in law to a normal variable of mean 0 and variance
	\[2\int_0^{s}ds'\int_0^{s'}ds''\int_{D_{s''}}\psi_{s'}\psi_{s''}(1-\sigma_{s''}).\]
	where $\psi_t$ solves the Laplace problem on $D_t$ with boundary values $\psi_t|_{\partial D_t}\equiv u|_{\partial D_t}$.
\end{theorem}
As before, if $u,v\in C_0^4(\mathbb{R}^d)$ with $\supp u,\supp v\subset D_s$, the variables $(L^s_m,u)$ and $(L^s_m,v)$ form a joint Gaussian random variable with covariance
\begin{equation}\label{covariance2}
	\int_0^{s}ds'\int_0^{s'}ds''\int_{D_{s''}}(\psi_{s'}\varphi_{s''}+\psi_{s''}\varphi_{s'})(1-\sigma_{s''}).
\end{equation}
After proving these results in the following sections, we will turn to an interesting application of Theorem \ref{orderedtime}. Namely, we will use Equations \ref{covariance1} and \ref{covariance2} to compute point-correlation functions, which encode the correlations between fluctuations at two different points. In some sense, point-correlation functions will be local versions of the above results.

%% file: 2_FixedTime.tex
\section{Proof of Theorem \ref{fixedtime}}\label{fixedtimeproof}
In our analysis below, we will make use of the grids
\[\mathcal{G}_m:=\left\{(x,y)\in\frac{1}{m}\mathbb{R}^2\;\bigg|\;x\in\frac{1}{m}\mathbb{Z}\;\text{or}\;y\in\frac{1}{m}\mathbb{Z}\right\}.\]
In particular, we will use the following notion of a \emph{grid harmonic function} on $\mathcal{G}_m$:
\begin{definition}
	A continuous function $\phi:U\subset\mathcal{G}_m\to\mathbb{R}$ is \emph{grid harmonic} if
	\[\Delta_h\phi(z):=\frac{m^2}{4}\left(\phi(z+1/m)+\phi(z-1/m)+\phi(z+i/m)+\phi(z-i/m)\right)-m^2\phi(z)=0\]
	on the nodes $z\in U\cap\frac{1}{m}\mathbb{Z}^2$, and $\phi$ is linear on each edge of $\mathcal{G}_m$.
\end{definition}
Finally, we will let $\mathscr{F}_{m,t}$ be the filtration generated by $A_{m}(t)$.
\begin{proof}[Proof of Theorem \ref{fixedtime}, Step 1]\let\qed\relax
	We will first relate $(E^s_m,u)$ to a family of martingales and show that the difference converges in law to zero.
	
	Let $\eps_m=Cm^{-3/5}$, as in Lemmas \ref{narrow} and \ref{narrowsand}, and let $\psi_m$ be harmonic on $(D_s)^{2\eps_m}$ with boundary values $u|_{\partial (D_s)^{2\eps_m}}$. Let $\psi_{(m)}$ solve the corresponding grid-Laplace problem on $\mathcal{G}_m\cap (D_s)^{2\eps_m}$. That is, $\psi_{(m)}$ is grid-harmonic in $\mathcal{G}_m\cap (D_s)^{2\eps_m}$, and 
	\[\psi_{(m)}|_{\partial (D_s)^{2\eps_m}}\equiv \psi_{m}|_{\partial (D_s)^{2\eps_m}}\equiv u|_{\partial (D_s)^{2\eps_m}}.\]
	Since $\psi_{(m)}=\psi_{m}$ on the boundary, standard estimates (for instance, see \cite[Theorem 3.5]{articleIfound}) give
	\begin{equation}\label{discharm1}
		\|\psi_{(m)}-\psi_m\|_\infty\leq C_1/m^2,
	\end{equation}
	where $C_1\sim \|\nabla^4 u\|$. Next, define the martingales
	\[M_m(t):=m^{-1}\left(\sum_{x\in A_{m}(t\wedge\tau^*)\setminus D_0}\psi_{(m)}(x)-\sum_{i=1}^{t\wedge\tau^*}\psi_{(m)}(z_{m,i})\right)=m^{-1}\sum_{A_{m}(t\wedge\tau^*)}\psi_{(m)}\cdot(1-\sigma_{m,t\wedge\tau^*}),\]
	where $\tau^*$ is the first time that $A_{m}(\tau^*)\not\subset (D_s)^{\eps_m}$. Note that $A_m(\tau^*)\subset (D_s)^{2\eps_m}$, so the function $\psi_{(m)}$ is defined on all of $A_m(\tau^*)$.
	
	Consider the event $\Eps$ that $\supp E^s_m \subset F^s_{m}$; by Lemmas \ref{narrow} and \ref{narrowsand}, this event occurs with probability $1-e^{-m^{2/5}}\upto 1$. In this case, $\tau^*\geq m^2s$, so---since $\psi_{(m)}$ is discrete harmonic---we have $M_m(m^2s)=(E^s_m,\psi_{(m)})$. To relate this to $(E^s_m,u)$, we first want to bound $\sup\nolimits_{F^s_{m}}|u-\psi_m|$. Suppose $x\in F^s_m$ achieves this supremum, and choose $x'\in\partial (D_s)^{2\eps_m}$ such that $|x-x'|\leq 4\eps_m$. Now, $\partial_i\psi_m$ solves the Laplace equation with boundary values $\partial_iu|_{\partial (D_s)^{2\eps_m}}$, so by the maximum principle, 
	\[|\partial_i\psi_m|\leq\sup\nolimits_{\partial D^{2\eps_m}}|\partial_iu|\leq \sup\nolimits_{\left(\bigcup_m(D_s)^{2\eps_m}\right)}|\nabla u|.\]
	In particular, $|\nabla\psi_m|\leq C_1=C_1(u)$ for all $m$, choosing a larger $C_1$ if necessary. Without loss of generality, we can take $C_1\geq \sup_{\left(\bigcup_mD^{\eps_m}\right)}|\nabla u|$. This implies that
	\begin{align*}
		\sup\nolimits_{F^s_m}|u-\psi_m|&=|u(x)-\psi_m(x)|\\
		&\leq|u(x)-u(x')|+|u(x')-\psi_m(x')|+|\psi_m(x')-\psi_m(x)|\\
		&=|u(x)-u(x')|+|\psi_m(x')-\psi_m(x)|\\
		&\leq 8\eps_mC_1.
	\end{align*}
	By (\ref{discharm1}), this means $\sup_{F^s_m}|u-\psi_{(m)}|=O(\eps_m)$. Thus, we find
	\begin{align}\label{justusedinconclusion}
		\left|(E^s_m,u)-(E^s_m,\psi_{(m)})\right|\leq m\cdot\op{Vol}(F^s_m)\sup\nolimits_{F^s_m}|u-\psi_{(m)}|=O(m\eps_m^2)=O(m^{-1/5}),
	\end{align}
	which converges to zero.
	
	Now, for any $\delta>0$, we can choose $m_0>0$ such that 
	\[\left|(E^s_m,u)-M_m(m^2s)\right|=\left|(E^s_m,u)-(E^s_m,\psi_{(m)})\right|\leq\delta\]
	on event $\Eps$ for any $m>m_0$. The probability of $\Eps^c$ tends to zero, so we know that $(E^s_m,u)-M_m(m^2s)$ converges in probability (and thus in law) to zero.
\end{proof}

\begin{proof}[Step 2]
	Now, we will show that the family of random variables $M_m(m^2s)$ converges in law to a zero-mean normal variable with variance $\int_{D_s}|\psi|^2(1-\sigma_s)$.
	
	For this step, we will follow the style of proof in \cite{eli16}. Define
	\[X_{m,t}=M_m(t)-M_m(t-1)=\begin{cases}m^{-1}(\psi_{(m)}(A_{m}(t)\setminus A_{m}(t-1))-\psi_{(m)}(z_{m,t})) & t\leq\tau^*\\
		0 & t>\tau^*
	\end{cases}.\]
	This is a mean-zero martingale difference array adapted to $\mathscr{F}_{m,t}$. The martingale central limit theorem stated in \cite[Theorem 3.2]{HALL198051} thus states that $M_m(m^2s)=\sum_{t\leq m^2s} X_{m,t}$ converges in law to a normal variable of mean 0 and variance $\int_{D_s}|\psi|^2(1-\sigma_s)$, so long as the following three conditions hold:
	\begin{enumerate}
		\item $\mathbb{E}\left[\max\nolimits_t |X_{m,t}|^2\right]$ is bounded in $m$. This also implies that the array is square-integrable, which is one of the hypotheses of the theorem.
		\item $\max\nolimits_t|X_{m,t}|\to 0$ in probability as $m\to\infty$.
		\item $\sum\nolimits_{t\leq m^2s} |X_{m,t}|^2\to \int_{D_s}|\psi|^2(1-\sigma_s)$ in probability as $m\to\infty$.
	\end{enumerate}
	As in [Sadovnik], we will handle the first two conditions by showing that $\mathbb{E}[\max_t|X_{m,t}|^a]\to 0$ for $a\geq 1$. This is clear from the following estimate:
	\begin{align*}|X_{m,t}|^a&=m^{-a}\left|\psi_{(m)}(A_{m}(t)\setminus A_{m}(t-1))-\psi_{(m)}(z_{m,t})\right|^a\\
		&\leq 2^am^{-a}\sup|\psi_{(m)}|^a\\
		&\leq 2^am^{-a}\sup\nolimits_{\left(\bigcup_mD^{\eps_m}\right)}|u|^a,
	\end{align*}
	from the maximum principle.
	
	For the final condition, define the random variables 
	\[S_m(t)=\sum_{i=1}^t|X_{m,i}|^2,\qquad Z_m(t)= m^{-2}\sum_{x\in A_{m}(t\wedge\tau^*)\setminus D_0}|\psi_{(m)}(x)|^2- m^{-2}\sum_{i=1}^{t\wedge\tau^*}|\psi_{(m)}(z_{m,i})|^2,\]
	\[N_m(t)=S_m(t)-Z_m(t).\]
	Our goal is to show that $N_m(m^2s)\to 0$ in probability, and thus that $S_m(t)$ can be well-approximated by the simpler variable $Z_m(t)$.
	
	For this, first note that $N_m$ satisfies the martingale property; we only need show this for time intervals before $\tau^*$, as $N_m$ remains constant thereafter. For $t\leq\tau^*$,
	\begin{align*}
		\mathbb{E}[&N_m(t)-N_m(t-1)|\mathscr{F}_{m,t-1}]\\
		&=\mathbb{E}\left[|X_{m,i}|^2-m^{-2}\left(\psi_{(m)}(A_{m}(t)\setminus A_{m}(t-1))^2-\psi_{(m)}(z_{m,t})^2\right)\big|\mathscr{F}_{m,t-1}\right]\\
		&=\mathbb{E}\bigg[m^{-2}\left(\psi_{(m)}(A_{m}(t)\setminus A_{m}(t-1))-\psi_{(m)}(z_{m,t})\right)^2\\
		&\qquad-m^{-2}\left(\psi_{(m)}(A_{m}(t)\setminus A_{m}(t-1))^2-\psi_{(m)}(z_{m,t})^2\right)\bigg|\mathscr{F}_{m,t-1}\bigg]\\
		&=\mathbb{E}\left[2m^{-2}\psi_{(m)}(z_{m,t})^2-2m^{-2}\psi_{(m)}(A_{m}(t)\setminus A_{m}(t-1))\psi_{(m)}(z_{m,t})\big|\mathscr{F}_{m,t-1}\right]\\
		&=2m^{-2}\psi_{(m)}(z_{m,t})\mathbb{E}\left[\psi_{(m)}(z_{m,t})-\psi_{(m)}(A_{m}(t)\setminus A_{m}(t-1))\big|\mathscr{F}_{m,t-1}\right]\\
		&=0.
	\end{align*}
	Since $S_m(1)=Z_m(1)=0$ (as $A_{m}(1)=\{z_{m,1}\}$ certainly), we have $N_m(1)=0$ and thus
	\[\mathbb{E}[N_m(m^2s)^2]=\mathbb{E}[\left(N_m(m^2s)-N_m(1)\right)^2]=\sum_{t=1}^{\lfloor m^2s\rfloor}\mathbb{E}[\left(N_m(t)-N_m(t-1)\right)^2]\]
	from the martingale property. Again taking $t\leq\tau^*$, we estimate
	\begin{align*}
		\mathbb{E}[\left(N_m(t)-N_m(t-1)\right)^2]&\leq 2\mathbb{E}[\left(S_m(t)-S_m(t-1)\right)^2]+2\mathbb{E}[\left(Z_m(t)-Z_m(t-1)\right)^2]\\
		&\leq 2\mathbb{E}\left[|X_{m,t}|^4\right]+2m^{-4}\mathbb{E}\left[\left(|\psi_{(m)}(A_{m}(t)\setminus A_{m}(t-1))|^2+|\psi_{(m)}(z_{m,t})|^2\right)^2\right]\\
		&\leq 8m^{-4}\sup|\psi_{(m)}|^4\\
		&\leq C_1m^{-4},
	\end{align*}
	where $C_1=C_1(u)$. This implies
	\[\mathbb{E}[N_m(m^2s)^2]=\sum_{t=1}^{\lfloor m^2s\rfloor}\mathbb{E}[\left(N_m(t)-N_m(t-1)\right)^2]\leq m^2s\cdot C_1m^{-4}=O(m^{-2}).\]
	Thus, $N_m(m^2s)\to 0$ in the $L^2$ norm, and thus also in probability.
	
	Finally, we show that $Z_m(m^2s)\to\int_{D_s}|\psi|^2(1-\sigma_s)$ in probability. From the above argument, this would imply that $S_m(m^2s)\to\int_{D_s}|\psi|^2(1-\sigma_s)$ in probability, which is exactly the third condition of the martingale central limit theorem.
	
	For this purpose, note that, on event $\Eps$ (where $\tau^*\geq m^2s$),
	\[Z_m(m^2s)=m^{-2}\sum_{A_{m}(m^2s)\setminus D_0}|\psi_{(m)}|^2- m^{-2}\sum_{i=1}^{\lfloor m^2s\rfloor}|\psi_{(m)}(z_{m,i})|^2=m^{-2}\sum_{x\in A_{m}(m^2s)}|\psi_{(m)}|^2(1-\sigma_{m,m^2s}).\]
	On Event $\Eps$, we know that $A_{m}(m^2s)$ differs from $D_s\cap\frac{1}{m}\mathbb{Z}^2$ by at most $O(m^2\eps_m)$ points; in this case,
	\[\bigg|Z_m(m^2s)-m^{-2}\sum_{ D_s\cap\frac{1}{m}\mathbb{Z}^2}|\psi_{(m)}|^2(1-\sigma_s)\bigg|=O(\eps_m),\]
	as $|\psi_{(m)}(x)|^2$ is uniformly bounded (as we saw above) in terms of $u$ and $\sum |\sigma_s-\sigma_{m,m^2s}|=O(1)$. In turn,
	\begin{align*}
		\bigg|m^{-2}\sum_{D_s\cap\frac{1}{m}\mathbb{Z}^2}|\psi_{(m)}|^2&(1-\sigma_s)-m^{-2}\sum_{ D_s\cap\frac{1}{m}\mathbb{Z}^2}|\psi_{m}|^2(1-\sigma_s)\bigg|\\
		&=\bigg|m^{-2}\sum_{D_s\cap\frac{1}{m}\mathbb{Z}^2}\left(|\psi_{(m)}|^2-|\psi_{m}|^2\right)(1-\sigma_s)\bigg|\\
		&=\bigg|m^{-2}\sum_{D_s\cap\frac{1}{m}\mathbb{Z}^2}\left(\psi_{(m)}-\psi_{m}\right)\left(\psi_{(m)}+\psi_{m}\right)(1-\sigma_s)\bigg|\\
		&=O(m^{-2}),
	\end{align*}
	using (\ref{discharm1}) in the final step. Now, we compare $m^{-2}\sum_{D_s\cap\frac{1}{m}\mathbb{Z}^2}|\psi_{m}|^2(1-\sigma_s)$ with $m^{-2}\sum_{D_s\cap\frac{1}{m}\mathbb{Z}^2}|\psi|^2(1-\sigma_s)$, where $\psi$ solves the Laplace equation on $D_s$ with $\psi|\partial D_s\equiv u|\partial D_s$. For this, suppose that $x\in\partial D_s$ maximizes $|\psi-\psi_m|$, and take $x'\in\partial (D_s)^{2\eps_m}$ such that $|x-x'|\leq 4\eps_m$. As in Step 1, choose $C_1$ such that $|\nabla u|,|\nabla\psi_m|\leq C_1$. Then we find
	\begin{align*}
		\sup\nolimits_{\partial D}|\psi-\psi_m|&=|\psi(x)-\psi_m(x)|\\
		&=|u(x)-\psi_m(x)|\\
		&\leq|u(x)-u(x')|+|u(x')-\psi_m(x')|+|\psi_m(x')-\psi_m(x)|\\
		&=|u(x)-u(x')|+|\psi_m(x')-\psi_m(x)|\\
		&\leq 8\eps_m C_1.
	\end{align*}
	Of course, $\psi-\psi_m$ is harmonic in $D_s$, so the maximum principle implies $\sup\nolimits_{ D_s}|\psi-\psi_m|\leq 8\eps_mC_1$. Arguing as before, we find
	\begin{align*}
		\bigg|m^{-2}\sum_{D_s\cap\frac{1}{m}\mathbb{Z}^2}|\psi_{m}|^2(1-\sigma_s)-m^{-2}\sum_{D_s\cap\frac{1}{m}\mathbb{Z}^2}|\psi|^2(1-\sigma_s)\bigg|=O(\eps_m).
	\end{align*}
	Putting these inequalities together shows that
	\[\bigg|Z_m(m^2s)-m^{-2}\sum_{D_s\cap\frac{1}{m}\mathbb{Z}^2}|\psi|^2(1-\sigma_s)\bigg|=O(\eps_n)\]
	on Event $\Eps$. Since $P(\Eps)\upto 1$, this implies that the above difference converges in probability to zero. Finally, $m^{-2}\sum_{D_s\cap\frac{1}{m}\mathbb{Z}^2}|\psi|^2(1-\sigma_s)$ converges to $\int_{D_s}|\psi|^2(1-\sigma_s)$, so the theorem is proved.
\end{proof}

%% file: 3_OrderedTime.tex
\section{Proof of Theorem \ref{orderedtime}}\label{orderedtimeproof}
We prove a slight generalization of this result, in the case that $\supp u$ is not necessarily contained in $D_s$:
\begin{lemma}
	Suppose $u\in C^4(\mathbb{R}^2)$. The random variables $(L^s_m,u)$ converge in law to a normal variable of mean 0 and variance
	\begin{equation}\label{complicated}
		s^2\int_{D_{s}}|\psi_{s}|^2(1-\sigma_{s})+2\int_0^{s}ds'\int_0^{s'}ds''\int_{D_{s'}}\psi_{s'}\psi_{s''}(1-\sigma_{s''})-2s\int_0^{s}ds'\int_{D_{s'}}\psi_{s}\psi_{s'}(1-\sigma_{s'}),
	\end{equation}
	where $\psi_t$ solves the Laplace problem on $D_t$ with boundary values $\psi|_{\partial D_t}\equiv u|_{\partial D_t}$.
\end{lemma}

\begin{remark}
	In the case of interest, with $\supp u\subset D_s$, we have that $\psi_s\equiv 0$ and thus that the above variance becomes
	\[2\int_0^{s}ds'\int_0^{s'}ds''\int_{D_{s'}}\psi_{s'}\psi_{s''}(1-\sigma_{s''}).\]
\end{remark}

\begin{proof}[Step 1]\let\qed\relax
	We first want to replace $(L^s_m,u)$ with a suitable martingale. Let $\psi^t_{m}$ solve the Dirichlet problem for $u$ on $(D_t)^{2\eps_m}$, with $\eps_m=Cm^{-3/5}$. Let $\psi^t_{(m)}$ solve the corresponding grid-Laplace problem on $\mathcal{G}_m\cap (D_t)^{2\eps_m}$. As in the proof of Theorem \ref{fixedtime}, this means that
	\begin{equation}\label{discharm2}
		\|\psi^t_{(m)}-\psi^t_m\|_\infty\leq C_1/m^2,
	\end{equation}
	where $C_1\sim \|\nabla^4 u\|$. Now define the martingale
	\begin{align*}
	M_m(t)=sm^{-1}\sum_{j=1}^{t\wedge\tau^*}&\left(\psi_{(m)}^{s}(A_{m}(j)\setminus A_{m}(j-1))-\psi_{(m)}^{s}(z_{m,j})\right)\\
	&-m^{-3}\sum_{\ell=1}^{\lfloor m^2s\rfloor}\sum_{j=1}^{\ell\wedge \tau_{\ell}\wedge t}\left(\psi_{(m)}^{\ell/m^2}(A_{m}(j)\setminus A_{m}(j-1))-\psi_{(m)}^{\ell/m^2}(z_{m,j})\right),
	\end{align*}
	where $\tau_\ell$ is first time that $A_{m}(j)$ exits $(D_{\ell/m^2})^{\eps_m}$, and $\tau^*:=\tau_{m^2s}$ is the first time that it exits $(D_{s})^{\eps_m}$.
	
	Consider the event $\Eps$, in which $\frac{1}{m}\mathbb{Z}^2\cap (D_{\ell/m^2})_{\eps_m}\subset A_{m}(\ell)\subset (D_{\ell/m^2})^{\eps_m}$ for all $\ell\leq m^2s$. By Lemma \ref{narrow}, this occurs with probability $1-e^{-m^{2/5}}\upto 1$. On this event, $\tau_\ell\geq\ell$ for all $\ell$, and
	\begin{align*}
	M_m(m^2s)&=sm^{-1}\sum\nolimits_{\frac{1}{m}\mathbb{Z}^2}\psi_{(m)}^{s}\cdot(\mathbf{1}_{A_{m}(m^2s)}-\sigma_{m,m^2s})-m^{-3}\sum_{\ell=1}^{\lfloor m^2s\rfloor}\sum\nolimits_{\frac{1}{m}\mathbb{Z}^2 }\psi_{(m)}^{\ell/m^2}\cdot(\mathbf{1}_{A_{m}(\ell)}-\sigma_{m,\ell})\\
	&=sm^{-1}\sum\nolimits_{\frac{1}{m}\mathbb{Z}^d}\psi_{(m)}^{s}\cdot(\mathbf{1}_{A_{m}(m^2s)}-\nu_{m,m^2s})-m^{-3}\sum_{\ell=1}^{\lfloor m^2s\rfloor}\sum\nolimits_{\frac{1}{m}\mathbb{Z}^2}\psi_{(m)}^{\ell/m^2}\cdot(\mathbf{1}_{A_{m,\ell}}-\nu_{m,\ell}).
	\end{align*}
	Of course, on event $\Eps$, the function $\mathbf{1}_{A_{m}(\ell)}-\nu_{m,\ell}$ is supported on $F_{m}^{\ell/m^2}=(D_{\ell/m^2})^{\eps_m}\setminus(D_{\ell/m^2})_{\eps_m}$; this set has volume $O(\eps_m)$, and $\sup_{F_{m}^{\ell/m^2}}\left|u-\psi^{\ell/m^2}_{(m)}\right|=O(\eps_m)$ as in the previous proof. Then we have
	\begin{align}\label{justusedinconclusion2}
	\begin{split}
		M_m(m^2s)&=sm^{-1}\sum_{\frac{1}{m}\mathbb{Z}^2}u\cdot(\mathbf{1}_{A_{m}(m^2s)}-\nu_{m,m^2s})-m^{-3}\sum_{\ell=1}^{\lfloor m^2s\rfloor}\sum_{\frac{1}{m}\mathbb{Z}^2}u\cdot(\mathbf{1}_{A_{m}(\ell)}-\nu_{m,\ell})+O(m\eps_m^2)\\
		&=m^{-3}\sum_{\ell=1}^{\lfloor m^2s\rfloor}\sum_{\frac{1}{m}\mathbb{Z}^2}\ell u\cdot(\mathbf{1}_{A_{m}(\ell)}- \mathbf{1}_{A_{m}(\ell-1)})-m^{-3}\sum_{\ell=1}^{\lfloor m^2s\rfloor}\sum_{\frac{1}{m}\mathbb{Z}^2}\ell u\cdot(\nu_{m,\ell}-\nu_{m,\ell-1})+O(m^{-1/5})\\
		&=(L^s_m,u)+O(m^{-1/5}).
	\end{split}
	\end{align}
	Thus, $M_m(m^2s)-(L^s_m,u)$ converges to zero in probability.
\end{proof}
\begin{proof}[Step 2]
	Note that the martingale intervals $X_{m,t}=M_m(t)-M_m(t-1)$ take the following form:
	\begin{align*}
		X_{m,t}=sm^{-1}\mathbf{1}_{\{t\leq\tau^*\}}\cdot&\left(\psi_{(m)}^{s}(A_{m}(t)\setminus A_{m}(t-1))-\psi_{(m)}^{s}(z_{m,t})\right)\\
		&-m^{-3}\sum_{\substack{t\leq\ell\leq m^2s\\\text{s.t.}\; \tau_\ell\geq t}}\left(\psi_{(m)}^{\ell/m^2}(A_{m}(t)\setminus A_{m}(t-1))-\psi_{(m)}^{\ell/m^2}(z_{m,t})\right).
	\end{align*}
	Now we need to show that $M_m(m^2s)$ approaches the appropriate normal distribution. We will again make use of the martingale central limit theorem \cite[Theorem 3.2]{HALL198051}---namely, our result is proved if we can show the following three conditions:
	\begin{enumerate}
		\item $\mathbb{E}\left[\max\nolimits_t |X_{m,t}|^2\right]$ is bounded in $m$.
		\item $\max\nolimits_t|X_{m,t}|\to 0$ in probability as $m\to\infty$.
		\item $\sum\nolimits_t |X_{m,t}|^2$ converges to the expression in (\ref{complicated}) in probability as $m\to\infty$.
	\end{enumerate}
	The first and second conditions follow from the following calculation, that $\mathbb{E}[\max_t|X_{m,t}|^a]\to 0$ for $a\geq 1$.
	\begin{align*}
	|X_{m,t}|^a&\leq 2^{a-1}s^am^{-a}\left|\psi_{(m)}^{s}(A_{m}(t)\setminus A_{m}(t-1))-\psi_{(m)}^{s}(z_{m,t})\right|^a\\
	&\qquad+2^{a-1}s^am^{-a}\sup_{\ell\geq t}\left|\psi_{(m)}^{\ell/m^2}(A_{m}(t)\setminus A_{m}(t-1))-\psi_{(m)}^{\ell/m^2}(z_{m,t})\right|^a\\
	&\leq 2^{a+1}s^am^{-a}\sup\left|\psi_{(m)}\right|^a\\
	&=O(m^{-a}),
	\end{align*}
	which proves the first two conditions.
	For the final condition, we again define auxiliary variables
	\begin{align*}
	Z_m(t)&= s^2m^{-2}\sum\nolimits_{\frac{1}{m}\mathbb{Z}^2}|\psi^{s}_{(m)}|^2(\mathbf{1}_{A_{m}(t\wedge\tau^*)}-\sigma_{m,t\wedge\tau^*})\\
	&\qquad+m^{-6}\sum_{1\leq j, k\leq m^2s}\sum\nolimits_{\frac{1}{m}\mathbb{Z}^2}\psi^{k/m^2}_{(m)}\psi^{j/m^2}_{(m)}(\mathbf{1}_{A_{m}(j\wedge\tau_j\wedge k\wedge\tau_k\wedge t)}-\sigma_{m,j\wedge\tau_j\wedge k\wedge\tau_k\wedge t})\\
	&\qquad-2sm^{-4}\sum_{1\leq j\leq m^2s}\sum\nolimits_{\frac{1}{m}\mathbb{Z}^2}\psi^{s}_{(m)}\psi^{j/m^2}_{(m)}(\mathbf{1}_{A_{m}(j\wedge\tau_j\wedge t)}-\sigma_{m,j\wedge\tau_j\wedge t}),
	\end{align*}
	\[S_m(t)=\sum_{j=1}^t|X_{m,j}|^2,\qquad N_m(t)=S_m(t)-Z_m(t).\]
	As before, $N_m$ satisfies the martingale property. To see this, we first factor the intervals of $Z_m$ as follows; below, write $a_{m,t}:=A_m(t)\setminus A_m(t-1)$ for the $t^{th}$ point joined to our IDLA.
	\begin{align*}
		Z_m(t)-Z_m(t-1)&=s^2m^{-2}\mathbf{1}_{\{\tau^*\geq t\}}\cdot\left(\psi^{s}_{(m)}(a_{m,t})^2-\psi^{s}_{(m)}(z_{m,t})^2\right)\\
		&\qquad+m^{-6}\sum_{\substack{t\leq j,k\leq m^2s\\\text{s.t.}\; \tau_j,\tau_k\geq t}}\left(\psi^{k/m^2}_{(m)}(a_{m,t})\psi^{j/m^2}_{(m)}(a_{m,t})-\psi^{k/m^2}_{(m)}(z_{m,t})\psi^{j/m^2}_{(m)}(z_{m,t})\right)\\
		&\qquad-2sm^{-4}\sum_{\substack{t\leq j\leq m^2s\\\text{s.t.}\; \tau_j\geq t}}\left(\psi^{s}_{(m)}(a_{m,t})\psi^{j/m^2}_{(m)}(a_{m,t})-\psi^{s}_{(m)}(z_{m,t})\psi^{j/m^2}_{(m)}(z_{m,t})\right)\\
		&=\bigg[sm^{-1}\mathbf{1}_{\{\tau^*\geq t\}}\cdot\psi^{s}_{(m)}(a_{m,t})-m^{-3}\sum_{\substack{t\leq j\leq m^2s\\\text{s.t.}\; \tau_j\geq t}}\psi^{j/m^2}_{(m)}(a_{m,t})\bigg]^2\\
		&\qquad-\bigg[sm^{-1}\mathbf{1}_{\{\tau^*\geq t\}}\cdot\psi^{s}_{(m)}(z_{m,t})-m^{-3}\sum_{\substack{t\leq j\leq m^2s\\\text{s.t.}\; \tau_j\geq t}}\psi^{j/m^2}_{(m)}(z_{m,t})\bigg]^2.
	\end{align*}
	We thus see that the only remaining terms of $N_m(t)-N_m(t-1)=|X_{m,t}|^2-(Z_m(t)-Z_m(t-1))$ are the following cross-terms, from which the martingale property follows:
	\begin{align*}
	\mathbb{E}[&N_m(t)-N_m(t-1)|\mathscr{F}_{m,t-1}]\\
	&=\mathbb{E}\bigg[|X_{m,i}|^2-(Z_m(t)-Z_m(t-1))\bigg|\mathscr{F}_{m,t-1}\bigg]\\
	&=\mathbb{E}\bigg[2\bigg[sm^{-1}\mathbf{1}_{\{\tau^*\geq t\}}\cdot\psi^{s}_{(m)}(z_{m,t})-m^{-3}\sum_{\substack{t\leq j\leq m^2s\\\text{s.t.}\; \tau_j\geq t}}\psi^{j/m^2}_{(m)}(z_{m,t})\bigg]^2\\
	&\qquad-2\bigg[sm^{-1}\mathbf{1}_{\{\tau^*\geq t\}}\cdot\psi^{s}_{(m)}(a_{m,t})-m^{-3}\sum_{\substack{t\leq j\leq m^2s\\\text{s.t.}\; \tau_j\geq t}}\psi^{j/m^2}_{(m)}(a_{m,t})\bigg]\\
	&\qquad\quad\times\bigg[sm^{-1}\mathbf{1}_{\{\tau^*\geq t\}}\cdot\psi^{s}_{(m)}(z_{m,t})-m^{-3}\sum_{\substack{t\leq j\leq m^2s\\\text{s.t.}\; \tau_j\geq t}}\psi^{j/m^2}_{(m)}(z_{m,t})\bigg]\;\bigg|\mathscr{F}_{m,t-1}\bigg]\\
	&\propto\mathbb{E}\bigg[\bigg[sm^{-1}\mathbf{1}_{\{\tau^*\geq t\}}\cdot\psi^{s}_{(m)}(z_{m,t})-m^{-3}\sum_{\substack{t\leq j\leq m^2s\\\text{s.t.}\; \tau_j\geq t}}\psi^{j/m^2}_{(m)}(z_{m,t})\bigg]\\
	&\qquad-\bigg[sm^{-1}\mathbf{1}_{\{\tau^*\geq t\}}\cdot\psi^{s}_{(m)}(a_{m,t})-m^{-3}\sum_{\substack{t\leq j\leq m^2s\\\text{s.t.}\; \tau_j\geq t}}\psi^{j/m^2}_{(m)}(a_{m,t})\bigg]\;\bigg|\mathscr{F}_{m,t-1}\bigg]\\
	&=0,
	\end{align*}
	using the fact that the last variable is a linear combination of martingale intervals adapted to $\mathscr{F}_{m,t}$. 
	As in the proof of Theorem \ref{fixedtime}, we can use this martingale property to show that---since $\left(S_m(t)-S_m(t-1)\right)^2$ and $\left(Z_m(t)-Z_m(t-1)\right)^2$ are of order $m^{-4}$---we have $\mathbb{E}[N_m(m^2s)^2]=O(m^{-2})$ and thus know that $N_m(m^2s)\to 0$ in probability.
	
	Finally, on event $\Eps$, we estimate $Z_m(m^2s)$ as follows:
	\begin{align*}
	Z_m(m^2s)&= s^2m^{-2}\sum\nolimits_{\frac{1}{m}\mathbb{Z}^d}|\psi^{s}_{(m)}|^2(\mathbf{1}_{A_{m}(m^2s)}-\sigma_{m,m^2s})\\
	&\qquad+m^{-6}\sum_{1\leq j, k\leq m^2s}\sum\nolimits_{\frac{1}{m}\mathbb{Z}^2}\psi^{k/m^2}_{(m)}\psi^{j/m^2}_{(m)}(\mathbf{1}_{A_{m}(j\wedge k)}-\sigma_{m,j\wedge k})\\
	&\qquad-2sm^{-4}\sum_{1\leq j\leq m^2s}\sum\nolimits_{\frac{1}{m}\mathbb{Z}^2}\psi^{s}_{(m)}\psi^{j/m^d}_{(m)}(\mathbf{1}_{A_{m}(j)}-\sigma_{m,j})\\
	&= s^2m^{-2}\sum\nolimits_{\frac{1}{m}\mathbb{Z}^2}|\psi^{s}_{(m)}|^2(\mathbf{1}_{A_{m}(m^2s)}-\sigma_{s})\\
	&\qquad+2m^{-6}\sum_{1\leq j\leq k\leq m^2s}\sum\nolimits_{\frac{1}{m}\mathbb{Z}^2}\psi^{k/m^2}_{(m)}\psi^{j/m^2}_{(m)}(\mathbf{1}_{A_{m}(j)}-\sigma_{j/m^2})\\
	&\qquad-2sm^{-4}\sum_{1\leq j\leq m^2s}\sum\nolimits_{\frac{1}{m}\mathbb{Z}^2}\psi^{s}_{(m)}\psi^{j/m^2}_{(m)}(\mathbf{1}_{A_{m}(j)}-\sigma_{j/m^2})+O(m^{-2}).
	\end{align*}
	Now, from (\ref{discharm2}), we can continue with the substitutions $\psi^\tau_{(m)}\to\psi^\tau_{m}$:
	\begin{align*}
	Z_m(m^2s)&= s^2m^{-2}\sum\nolimits_{\frac{1}{m}\mathbb{Z}^2}|\psi^{s}_{m}|^2(\mathbf{1}_{A_{m}(m^2s)}-\sigma_{s})\\
	&\qquad+2m^{-6}\sum_{1\leq j\leq k\leq m^2s}\sum\nolimits_{\frac{1}{m}\mathbb{Z}^2}\psi^{k/m^2}_{m}\psi^{j/m^2}_{m}(\mathbf{1}_{A_{m}(j)}-\sigma_{j/m^2})\\
	&\qquad-2sm^{-4}\sum_{1\leq j\leq m^2s}\sum\nolimits_{\frac{1}{m}\mathbb{Z}^2}\psi^{s}_{m}\psi^{j/m^2}_{m}(\mathbf{1}_{A_{m}(j)}-\sigma_{j/m^2})+O(m^{-2}).
	\end{align*}
	On event $\Eps$, the sets $A_m(j)$ and $D_{j/m^2}$ differ by at most $O(m^2\eps_m)$ points on the lattice $\frac{1}{m}\mathbb{Z}^2$, so we can replace $\mathbf{1}_{A_m(j)}\to\mathbf{1}_{D_{j/m^2}}$ with only an additional $O(\eps_m)$ error:
	\begin{align*}
	Z_m(m^2s)&= s^2m^{-2}\sum\nolimits_{\frac{1}{m}\mathbb{Z}^2}|\psi^{s}_{m}|^2(\mathbf{1}_{D_{s}}-\sigma_{s})\\
	&\qquad+2m^{-6}\sum_{1\leq j\leq k\leq m^2s}\sum\nolimits_{\frac{1}{m}\mathbb{Z}^2}\psi^{k/m^2}_{m}\psi^{j/m^2}_{m}(\mathbf{1}_{D_{j/m^2}}-\sigma_{j/m^2})\\
	&\qquad-2sm^{-4}\sum_{1\leq j\leq m^2s}\sum\nolimits_{\frac{1}{m}\mathbb{Z}^2}\psi^{s}_{m}\psi^{j/m^2}_{m}(\mathbf{1}_{D_{j/m^2}}-\sigma_{j/m^2})+O(\eps_m).
	\end{align*}
	Finally, since the derivatives of $\psi^\tau_m$ are uniformly bounded in both $m$ and $\tau$, we can swap these sums with the appropriate integrals with an error of $O(m^{-1})$ (which we wrap into the existing $O(\eps_m)$ term):
	\begin{align*}	
	Z_m(m^2s)&= s^2\int_{D_{s}}|\psi^{s}_{m}|^2(1-\sigma_{s})+2m^{-4}\sum_{1\leq j\leq k\leq m^2s}\int_{D_{j/m^2}}\psi^{k/m^2}_{m}\psi^{j/m^2}_{m}(1-\sigma_{j/m^2})\\
	&\qquad-2sm^{-2}\sum_{1\leq j\leq m^2s}\int_{D_{j/m^2}}\psi^{s}_{m}\psi^{j/m^2}_{m}(1-\sigma_{j/m^2})+O(\eps_m)\\
	&= s^2\int_{D_{s}}|\psi^{s}_{m}|^2(1-\sigma_{s})+2\int_0^{s}ds'\int_0^{s'}ds''\int_{D_{s''}}\psi^{s'}_{m}\psi^{s''}_{m}(1-\sigma_{s''})\\
	&\qquad-2s\int_0^{s}ds'\int_{D_{s'}}\psi^{s}_{m}\psi^{s'}_{m}(1-\sigma_{s'})+O(\eps_m)\\
	&= s^2\int_{D_{s}}|\psi_{s}|^2(1-\sigma_{s})+2\int_0^{s}ds'\int_0^{s'}ds''\int_{D_{s''}}\psi_{s'}\psi_{s''}(1-\sigma_{s''})\\
	&\qquad-2s\int_0^{s}ds'\int_{D_{s'}}\psi_{s}\psi_{s'}(1-\sigma_{s'})+O(\eps_m).
	\end{align*}
\end{proof}

%% file: 4_Correlations.tex
\section{Point correlation functions}\label{correlations}
In this section, we will compute \emph{point-correlation} functions for extended-source IDLA. In short, we want to find a local version of Equation \ref{covariance2}, which would tell us the correlation between IDLA fluctuations at two specific points, $p,q\in \op{int}(D_T)\setminus D_0$. We phrase this problem in terms of limits of smooth bump functions, which we already know how to handle from our main results. Fix $\eps>0$, and let $\eta_p^\eps$ and $\eta_q^\eps$ be smooth functions satisfying
\begin{equation}\label{requirements}
	\supp\eta_p^\eps\subset B_\eps(p),\qquad \supp\eta_q^\eps\subset B_\eps(q),\qquad \int\eta_p^\eps=\int\eta_q^\eps=1.
\end{equation}
Without loss of generality, we will assume that $B_\eps(p),B_\eps(q)\subset D_T$, and we will write $L_m:=L^T_m$ for the lateness function at time $T$. From Theorem \ref{orderedtime}, we know that $(L_m,\eta_p^\eps)$ [resp., $(L_m,\eta_q^\eps)$] tends to a Gaussian variable $L(\eta_p^\eps)$ [resp., $L(\eta_q^\eps)$] in $m$.

Our primary result is the following:
\begin{theorem}\label{correlationtheorem}
	Suppose $p,q\in \op{int}(D_T)\setminus D_0$, and $s_p$ and $s_q$ satisfy $p\in\partial D_{s_p}, q\in\partial D_{s_q}$. For $\eps>0$, further suppose that $\eta_q^\eps$ and $\eta_p^\eps$ are smooth functions satisfying (\ref{requirements}). The covariance between $L(\eta_q^\eps)=\lim_{m\to\infty}(L_m,\eta_q^\eps)$ and $L(\eta_p^\eps)=\lim_{m\to\infty}(L_m,\eta_p^\eps)$ satisfies
	\[g(p,q):=\lim_{\eps\to 0}\mathbb{E}\left[L(\eta_q^\eps)L(\eta_p^\eps)\right]=\frac{1}{v_pv_q}\int_{D_{s_*}}F_pF_q(1-\sigma_{s_*}),\]
	where $s_*=\min(s_p,s_q)$, $v_p$ and $v_q$ are the velocities of the flow $s\mapsto D_s$ at $p$ and $q$ (at times $s_p$ and $s_q$, respectively), and $F_p$ and $F_q$ are the Poisson kernels of $D_{s_p}$ and $D_{s_q}$ at $p$ and $q$, respectively.
\end{theorem}
For completeness' sake, we first recall the notion of a \emph{Poisson kernel}:
\begin{definition}
	Suppose $D$ is a smoothly bounded domain, and $p\in\partial D$. Then the Poisson kernel $F_{p,D}$ of $D$ at $p$ is the harmonic function on $\op{int}(D)$ satisfying 
	\begin{equation}\label{derivthing}
		F_p(x)=\pdrv{}{\mathbf{n}}G_D(x,\xi)|_{\xi=p},
	\end{equation}
	where $G_D(x,y)$ is the continuous Green's function for the domain $D$, and $\partial/\partial\mathbf{n}$ is the inward normal derivative with respect to the second variable.
	
	Importantly, if $f:C^0(\partial D)$, then the function
	\begin{equation}\label{poissonkerneleq}
		\phi(x)=\int_{\partial D}d\xi\;F_\xi(x)
	\end{equation}
	is harmonic on $D$ and satisfies $\phi|_{\partial D}\equiv f$.
	
	We will use these functions in the following context. If $p\in \op{int}(D_T)\setminus D_0$, there is a unique $s_p>0$ such that $p\in\partial D_{s_p}$. We write $F_p$ for the Poisson kernel of $D_{s_p}$ at $p$.
\end{definition}
\begin{proof}[Proof of Theorem \ref{correlationtheorem}]
	We first deal with the case that $s_p\neq s_q$, so that $p$ and $q$ are hit at different times by the flow $s\mapsto D_s$. Without loss of generality, suppose $s_p> s_q$, and suppose $\eps$ is small enough that 
	\begin{equation}\label{inequality}
		\inf\{s\;|\;B_\eps(p)\cap\partial D_s\neq\emptyset\}>\sup\{s\;|\;B_\eps(q)\cap\partial D_s\neq\emptyset\}.
	\end{equation}
	From (\ref{covariance2}), we know that
	\[\mathbb{E}\left[L(\eta_q^\eps)L(\eta_p^\eps)\right]=\int_0^{T}ds'\int_0^{s'}ds''\int_{D_{s''}}\psi^\eps_{s'}\varphi^\eps_{s''}(1-\sigma_{s''}),\]
	where $\psi^\eps_{s}$ and $\varphi^\eps_s$ are harmonic functions on $D_s$ satisfying $\psi^\eps_{s}|\partial D_s\equiv\eta_p^\eps$ and $\varphi^\eps_{s}|\partial D_s\equiv\eta_q^\eps$. In the above formula, we removed the $\psi^\eps_{s''}\varphi^\eps_{s'}$ term that appears in (\ref{covariance2}); these terms must all vanish, from (\ref{inequality}). Next, note that the remaining terms can only be nonzero when $s'$ [resp., $s''$] lies in a thin (i.e., $O(\eps)$) band around $s_p$ [resp, $s_q$]. Define
	\[s_-:=\inf\{s\;|\;B_\eps(q)\cap\partial D_s\neq\emptyset\}\]
	to be the smallest value of $s''$ such that $\varphi^\eps_{s'}$ is nonzero. From our above discussion, we can write
	\[\mathbb{E}\left[L(\eta_q^\eps)L(\eta_p^\eps)\right]=\int_0^{T}ds'\int_0^{s'}ds''\int_{D_{s_-}}\psi^\eps_{s'}\varphi^\eps_{s''}(1-\sigma_{s_-}) + O(\eps),\]
	also using the continuity of $s\mapsto\int\sigma_s$.
	Note that the third integral is now always taken over the same set. Now, introduce coordinates $(s,\theta)$ near $p$ such that $(s,\cdot)\in D_s$ and such that $\theta|\partial D_{s}$ measures the (signed) arclength from $(s,0)$ along $\partial D_s$. Introduce similar coordinates $(s,\alpha)$ near $q$. Without loss of generality, we assume $p=(s_p,0)$. From (\ref{poissonkerneleq}), we can rewrite
	\[\psi_s^\eps(z)=\int d\theta\; \eta_p^\eps(s,\theta)F_{(s,\theta)}(z),\qquad \varphi_s^\eps(z)=\int d\alpha\; \eta_q^\eps(s,\alpha)F_{(s,\alpha)}(z),\]
	which gives the following formula for the covariance:
	\[\mathbb{E}\left[L(\eta_q^\eps)L(\eta_p^\eps)\right]=\int_0^{T}ds'\int d\theta\;\eta_p^\eps(s',\theta)\int_0^{s'}ds''\int d\alpha\;\eta_q^\eps(s'',\alpha)\int_{D_{s_-}}F_{(s',\theta)}F_{(s'',\alpha)}(1-\sigma_{s_-}) + O(\eps).\]
	Now, $\eta_p^\eps$ is supported on an $\eps$-ball around $p$, so we only have to consider $F_{(s,\theta)}$ for $(s,\theta)\in B_\eps(p)$. For these points, we find that\footnote{For instance, we can find this estimate by first comparing $F_{(s,\theta)}$ and $F_p$ to nearby Green's functions using (\ref{derivthing}), and then comparing the Green's functions to one another by bounding their gradients above as $|\nabla_x G_D(x,y)|\leq C|x-y|^{-1}$.} $|F_{(s,\theta)}(z)-F_{p}(z)|\leq\frac{\eps C}{|z-p|^2}$, and thus
	\begin{align*}
		\int_{D_{s_p}}|F_{(s,\theta)}-F_{p}|=\int_{B_{\eps^{1/3}}(p)\cap D_{s_p}}&|F_{(s,\theta)}-F_{p}|+\int_{D_{s_p}\setminus B_{\eps^{1/3}}(p)}|F_{(s,\theta)}-F_{p}|\\
		&\leq\int_{D_{s_p}\setminus (D_{s_p})_{\eps^{1/3}}}\left(|F_{(s,\theta)}|+|F_{p}|\right) + \eps^{1/3} C\op{vol}(D_{s_p})=O(\eps^{1/3}).
	\end{align*}
	Repeating the same argument for $F_{(s,\alpha)}$ and $F_q$ (and using the fact that $F_p$ and $F_q$ are bounded near the pole of the other), we find
	\begin{align*}
		\mathbb{E}\left[L(\eta_q^\eps)L(\eta_p^\eps)\right]=\int_0^{T}ds'\int d\theta\;\eta_p^\eps(s',\theta)\int_0^{s'}ds''\int d\alpha\;\eta_q^\eps(s'',\alpha)\int_{D_{s_-}}F_{p}F_{q}(1-\sigma_{s_-}) + O(\eps^{1/3}).
	\end{align*}
	Finally, we convert from the coordinates $(s,\theta)$ and $(s,\alpha)$ back to standard Euclidean coordinates. For this, note that $(s,\theta)$ and $(s,\alpha)$ are orthogonal coordinate systems, and that $\theta$ and $\alpha$ are unit-speed parametrized, by definition. Thus, the only contributions to $|d(s,\theta)/d(x,y)|$ and $|d(s,\alpha)/d(x,y)|$ are the scaling factors in the $s$-direction. These are exactly the (inverse) velocities $v(s,\theta)^{-1}$ and $v(s,\alpha)^{-1}$ of the flow $s\mapsto D_s$, and we find
	\[\int ds'\int d\theta\;\eta_p^\eps(s',\theta)=\int dA\;v^{-1}\eta_p^\eps=v_p^{-1}\int dA\;\eta_p^\eps+O(\eps)=v_p^{-1}+O(\eps),\]
	and similarly $\int ds''\int d\alpha\;\eta_q^\eps(s'',\alpha)=v_q^{-1}+O(\eps)$. Putting these ingredients together, we get
	\begin{align*}
		\mathbb{E}\left[L(\eta_q^\eps)L(\eta_p^\eps)\right]=\frac{1}{v_pv_q}\int_{D_{s_-}}F_{p}F_{q}(1-\sigma_{s_-}) + O(\eps^{1/3})=\frac{1}{v_pv_q}\int_{D_{s_*}}F_{p}F_{q}(1-\sigma_{s_*}) + O(\eps^{1/3}),
	\end{align*}
	wrapping the $O(\eps)$ error term from switching $s_-$ to $s_*$ into the existing $O(\eps^{1/3})$ error.
	
	Now, assume that $s_p=s_q$, and suppose $\eps$ is small enough that $B_\eps(p)$ and $B_\eps(q)$ are disjoint. Let
	\[s_-:=\inf\left\{s\;|\;\left(B_\eps(q)\cup B_\eps(p)\right)\cap\partial D_s\neq\emptyset\right\},\]
	so that, as before,
	\[\mathbb{E}\left[L(\eta_q^\eps)L(\eta_p^\eps)\right]=\int_0^{T}ds'\int_0^{s'}ds''\int_{D_{s_-}}(\psi^\eps_{s'}\varphi^\eps_{s''}+\psi^\eps_{s''}\varphi^\eps_{s'})(1-\sigma_{s_-}) + O(\eps).\]
	We can split these terms as follows:
	\begin{align*}
		\int_0^{T}ds'\int_0^{s'}ds''&\int_{D_{s_-}}(\psi^\eps_{s'}\varphi^\eps_{s''}+\psi^\eps_{s''}\varphi^\eps_{s'})(1-\sigma_{s_-})\\
		&=\int_0^{T}ds'\int_0^{s'}ds''\int_{D_{s_-}}\psi^\eps_{s'}\varphi^\eps_{s''}(1-\sigma_{s_-})+\int_0^{T}ds'\int_0^{s'}ds''\int_{D_{s_-}}\psi^\eps_{s''}\varphi^\eps_{s'}(1-\sigma_{s_-})\\
		&=\int_0^{T}ds'\int_0^{s'}ds''\int_{D_{s_-}}\psi^\eps_{s'}\varphi^\eps_{s''}(1-\sigma_{s_-})+\int_0^{T}ds''\int_{s''}^{T}ds'\int_{D_{s_-}}\psi^\eps_{s''}\varphi^\eps_{s'}(1-\sigma_{s_-})\\
		&=\int_0^{T}ds'\int_0^{T}ds''\int_{D_{s_-}}\psi^\eps_{s'}\varphi^\eps_{s''}(1-\sigma_{s_-}).
	\end{align*}
	At this point, we can follow the same logic as in the first case, and the theorem follows.
\end{proof}
We can apply this formula concretely to the case of a radially-expanding disk. Suppose that $D_0$ is the unit disk, and set $Q_0^{s}=B_{\sqrt{s/\pi}}$ for $s\in[0,1)$. In this setting, we can imagine our source as a collection of outwardly moving rings of radius $0\leq r<1$, as shown in Figure \ref{diskfig}. From symmetry considerations, it is clear that $D_s=B_{\sqrt{1+s/\pi}}$ are outwardly expanding disks. 
\begin{figure}[H]
	\centering
	\begin{tikzpicture}
		\node[anchor=south west,inner sep=0] (image) at (0,0) {\includegraphics[scale=.3]{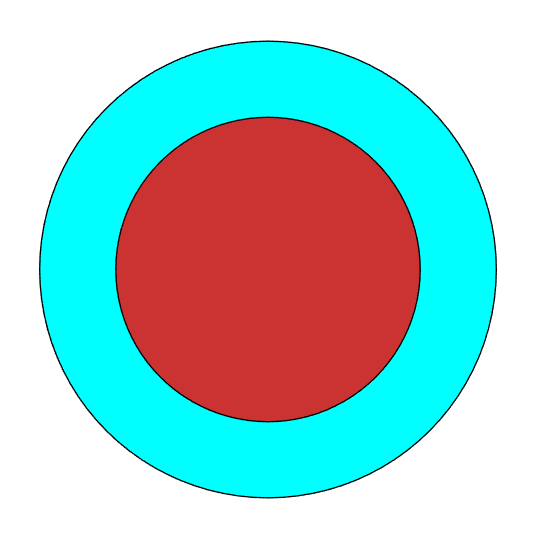}};
		\draw [->,
		line join=round,
		decorate, decoration={
			zigzag,
			segment length=15,
			amplitude=2,post=lineto,
			post length=2pt
		},ultra thick]  (4.7,2.2) -- (6.5,2.2);
		\node[anchor=south west,inner sep=0] (image) at (7,0) {\includegraphics[scale=.3]{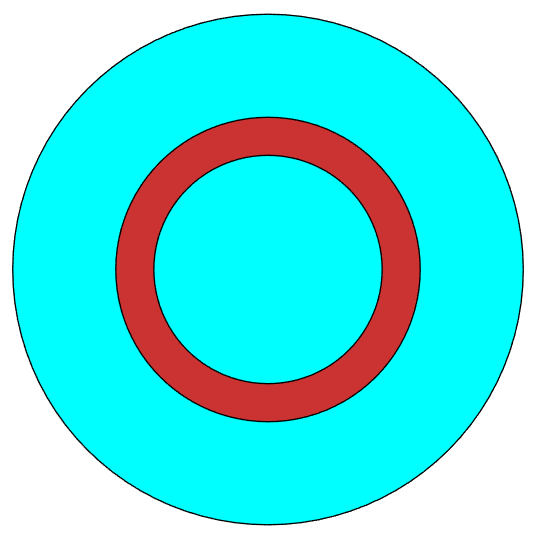}};
	\end{tikzpicture}%
	\caption{An illustration of the flow $s\mapsto D_s$ in the case of a radially-expanding disk. Here, $D_0$ is the unit disk (cyan, left), and our single source set $Q_0^T$ is a smaller disk within it (red, left). As time passes, source points move from the center of $Q_0^s$ to the outer boundary of $D_s$.}\label{diskfig}
\end{figure}
Suppose $p$ and $q$ lie in the plane, on origin-centered circles of radii $1<r_q\leq r_p<\sqrt{2}$ and at polar angles $\theta_p,\theta_q$. The functions $F_p$ and $F_q$ take the following forms:

\[F_p(re^{i\theta})=\sum_{n\in\mathbb{Z}}\left(r/r_p\right)^{|n|}e^{in(\theta-\theta_p)},\qquad F_q(re^{i\theta})=\sum_{n\in\mathbb{Z}}\left(r/r_q\right)^{|n|}e^{in(\theta-\theta_q)}.\]
Then, from Theorem \ref{correlationtheorem}, we can calculate
\begin{align*}
g(p,q)&=\frac{1}{v_pv_q}\int_{0}^{r_q}rdr\int_0^{2\pi}d\theta\;F_pF_q(1-\sigma_{\pi (r_q^2-1)})\\
&=(2\pi)^2r_pr_q\int_{0}^{r_q}dr\int_0^{2\pi}d\theta\;rF_pF_q(1-\sigma_{\pi (r_q^2-1)})\\
&=(2\pi)^2r_pr_q\int_{0}^{r_q}dr\int_0^{2\pi}d\theta\;r\sum_{j,k\in\mathbb{Z}}\frac{r^{|j|+|k|}}{r_p^{|j|}r_q^{|k|}}e^{i(j+k)\theta-ij\theta_p-ik\theta_q}(1-\sigma_{\pi (r_q^2-1)}).
\end{align*}
Only the terms with $j+k=0$ survive when integrating $\theta$:
\begin{align*}
g(p,q)&=(2\pi)^3r_pr_q\int_{0}^{r_q}dr\;\sum_{j\in\mathbb{Z}}\frac{r^{2|j|+1}}{r_p^{|j|}r_q^{|j|}}e^{ij(\theta_p-\theta_q)}(1-\sigma_{\pi (r_q^2-1)}).
\end{align*}
Now, we break this into two integrals using $\sigma_{\pi (r_q^2-1)}=\mathbf{1}_{\{r\leq\sqrt{r_q^2-1}\}}+\mathbf{1}_{\{r\leq1\}}$:
\begin{align*}
g(p,q)&=(2\pi)^3r_pr_q\int_{1}^{r_q}dr\;\sum_{j\in\mathbb{Z}}\frac{r^{2|j|+1}}{r_p^{|j|}r_q^{|j|}}e^{ij(\theta_p-\theta_q)}-(2\pi)^3r_pr_q\int_{0}^{\sqrt{r_q^2-1}}dr\;\sum_{j\in\mathbb{Z}}\frac{r^{2|j|+1}}{r_p^{|j|}r_q^{|j|}}e^{ij(\theta_p-\theta_q)}\\
&=(2\pi)^3r_pr_q\sum_{j\in\mathbb{Z}}\frac{e^{ij(\theta_p-\theta_q)}}{r_p^{|j|}r_q^{|j|}}\left(\int_{1}^{r_q}r^{2|j|+1}dr-\int_{0}^{\sqrt{r_q^2-1}}r^{2|j|+1}dr\right)\\
&=(2\pi)^3r_pr_q\sum_{j\in\mathbb{Z}}\frac{1}{2|j|+2}\frac{e^{ij(\theta_p-\theta_q)}}{r_p^{|j|}r_q^{|j|}}\left(r_q^{2|j|+2}-1-(r_q^2-1)^{|j|+1}\right).
\end{align*}
We can discard the negative frequency modes by rewriting this sum as twice the real part of its positive frequency modes:
\begin{align*}
g(p,q)&=(2\pi)^3(r_pr_q)^2\op{Re}\left(\sum_{j=0}^\infty\frac{1}{j+1}\frac{e^{ij(\theta_p-\theta_q)}}{r_p^{j+1}r_q^{j+1}}\left(r_q^{2j+2}-1-(r_q^2-1)^{j+1}\right)\right)\\
&=(2\pi)^3(r_pr_q)^2\op{Re}\left(e^{-i(\theta_p-\theta_q)}\sum_{j=0}^\infty\frac{1}{j+1}e^{i(j+1)(\theta_p-\theta_q)}\left(\frac{r_q^{j+1}}{r_p^{j+1}}-\frac{1}{r_p^{j+1}r_q^{j+1}}-\frac{(r_q^2-1)^{j+1}}{r_p^{j+1}r_q^{j+1}}\right)\right)\\
&=-(2\pi)^3(r_pr_q)^2\op{Re}\bigg[e^{-i(\theta_p-\theta_q)}\bigg(\op{Log}(1-e^{i(\theta_p-\theta_q)}r_q/r_p)-\op{Log}(1-e^{i(\theta_p-\theta_q)}1/r_pr_q)\\
&\hspace{3.3in}-\op{Log}(1-e^{i(\theta_p-\theta_q)}(r_q^2-1)/r_pr_q)\bigg)\bigg]\\
&=-(2\pi)^3|pq|\op{Re}\left[\ol{p}q\big(\op{Log}(1-\ol{q}/\ol{p})-\op{Log}(1-1/\ol{p}q)-\op{Log}(1-(|q|^2-1)/\ol{p}q)\big)\right],
\end{align*}
where $\op{Log}$ denotes the principle value of the logarithm, and we view $p$ and $q$ as complex numbers. This function is plotted in Figure \ref{pointcorrelationfig}.

\begin{figure}[H]
	\begin{subfigure}{.5\textwidth}
		\centering
		\includegraphics[scale=.3]{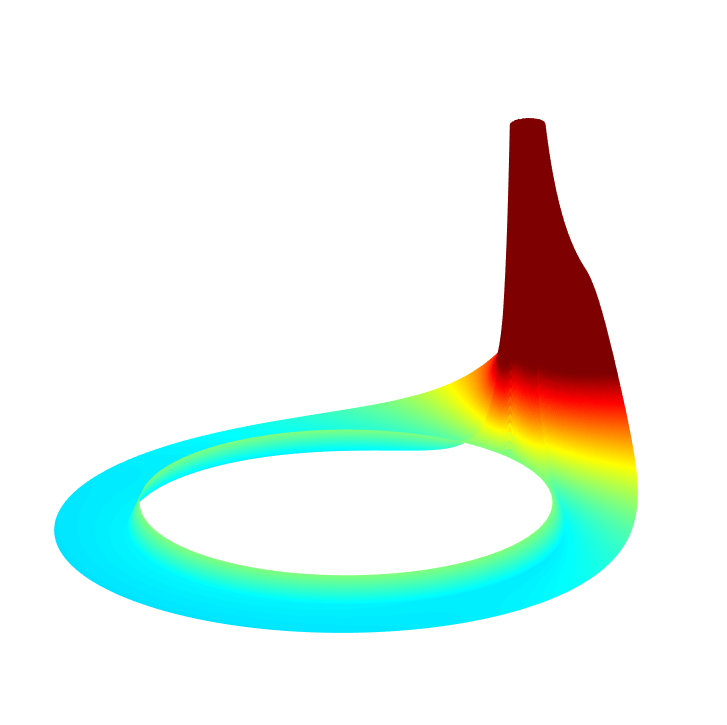}
	\end{subfigure}%
	\begin{subfigure}{.5\textwidth}
		\centering
		\includegraphics[scale=.3]{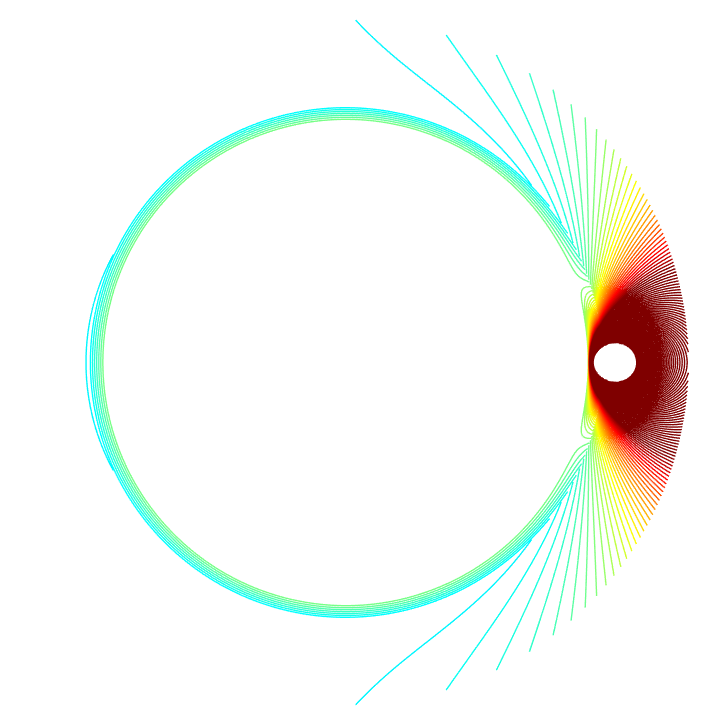}
	\end{subfigure}
	\caption{A plot and contour map of the function $g(p,q)$, where $q=(\nicefrac{1}{2}+\nicefrac{\sqrt{5}}{4},0)$ is fixed. Note the logarithmic singularity at $p=q$, and that the function vanishes for $p$ along the unit circle (the inner boundary of the domain).}\label{pointcorrelationfig}
\end{figure}
The function $g(p,q)$ has several key properties, which we can see in Figure \ref{pointcorrelationfig}. For one, $g(p,q)$ is positive if and only if $\theta_p$ and $\theta_q$ are nearby. This confirms the geometric intuition that, for instance, an early point leads to other nearby early points, but that it prevents distant early points (by using up particle mass itself). Secondly, $g(p,q)$ vanishes as either $p$ or $q$ approaches the unit disk, likely reflecting the fact that points nearer to $D_0$ are ``more deterministic''---i.e., that the variance of their lateness decreases to 0 as they get closer to $D_0$. It is easy to check from Theorem \ref{correlationtheorem} that this property holds true for other flows, as well, with the unit disk replaced by $D_0$ in general.

Next, note the logarithmic singularity present at $p=q$. This is to be expected, in analogy to the (free space) Green's function $G(x,y)=\log|x-y|$. Indeed, just as we can view the Green's function as giving an inner product 
\[(u,v)_{-1}:=\int dxdy\;u(x)G(x,y)v(y)=\int dy\;(\nabla^{-2}u)(y)v(y),\]
we can view the point-correlation function $g$ as the kernel of the inner product defined in Equation \ref{covariance2}:
\[(u,v)_g:=\int dxdy\;u(x)g(x,y)v(y)=\int_0^{T}ds'\int_0^{s'}ds''\int_{D_{s''}}(\psi_{s'}\varphi_{s''}+\psi_{s''}\varphi_{s'})(1-\sigma_{s''}),\]
where, as before, $\psi_s$ and $\varphi_s$ are the solutions of the Dirichlet problem on $D_s$ for $u$ and $v$, respectively.

%% file: 5_Conclusion.tex
\section{Directions for further research}

One interesting extension of this work would be to extend these results to higher dimensions. In dimension $d$, the appropriate scaling factor for the fluctuation functions $E^s_m$ and $L^s_m$ would be $m^{d/2}$ (just as it is $m=m^{2/2}$ here). In general, then, the error found in (\ref{justusedinconclusion}) and (\ref{justusedinconclusion2}) would come out to be $O(m^{d/2}\eps_m^2)$. For this to decrease, we then need the bound $\eps_m=o(m^{-d/4})$ on the maximum fluctuations. This is likely possible to achieve if $d=3$, but clearly impossible for $d\geq 4$.

However, there are weaker results that remain possible for $d\geq 4$. For one, if we require the test function $u$ to be harmonic, then we could achieve $\|u-\psi_{(m)}\|_\infty = O(m^{-2})$ on the domain of interest, rather than our existing $\|u-\psi_{(m)}\|_\infty = O(\eps_m)$. In this case, the requirement on $\eps_m$ becomes $\eps_m=o(m^{2-d/2})$, which now appears possible for dimensions 4 and 5.

Another important direction of research would be to generalize the sorts of possible sources for IDLA. For instance, it would be interesting to see if corresponding scaling limits hold if, instead of starting from a concentrated mass distribution, we were to start points evenly from a submanifold of $D_0$. Starting from the boundary of $D_0$, for example, may provide a good substitute for starting particles evenly across $D_0$ itself. In chemical applications, this adjusted setting could model a solid particle source of a particular shape.

Fortunately, the methods used in this paper translate fairly straightforwardly to other settings. The greatest obstacle to generalizing our results is finding an analogue to Lemma \ref{narrow}, which was the primary result of the preceding paper \cite{darrow2020convergence}. Indeed, if it could be shown that the fluctuations of IDLA from a particular source satisfy a similar $O(m^{-1/2-\eps})$ bound (for any $\eps>0$), the remainder of our argument could likely be repeated.

%% file: _Appendix.tex
\renewcommand{\thesection}{A}
\section{Appendix: maximum fluctuations of the divisible sandpile}\label{appendix}
We will use the capital $N_m(t)$ to denote the fully occupied set
\[N_m(t):=\{\nu_{m,t}=1\}\subset\frac{1}{m}\mathbb{Z}^2.\]
We will also use the notation of \cite{darrow2020convergence}---in particular, for any $\zeta\in\frac{1}{m}\mathbb{Z}^2\setminus D_0$, we will write $\tau=\tau(\zeta)$ for the time at which $\zeta\in\partial D_\tau$, and we will use $H_\zeta$ and $\Omega_\zeta$ exactly as in that paper. We will not give more details on these objects here.

Now, we say that a point $z\in\frac{1}{m}\mathbb{Z}^2$ is $\eps$-early at time $t$ if $z\in N_m(t)$, but $z\notin (D_{t/m^2})^\eps$. Similarly, $z$ is $\eps$-late at time $t$ if $z\in (D_{t/m^2})_\eps$, but $z\notin N_m(t)$.

Finally, we will define a stopped version of $\nu_{m,t}$, as follows:

\begin{definition}[Stopped Sandpile]
	Given $\nu_{\zeta,n}$, define the intermediate function $\nu_{\zeta,n}^0=\nu_{\zeta,n}+\mathbf{1}_{\{z_{m,n+1}\}}$. At each time step $t$, choose a point $z=z(t)\in\supp\nu_{\zeta,n}^t\setminus\partial D_\tau$ such that $\nu_{\zeta,n}^t(z)>1$. Let $W_{m,n}^{t}(s)$ be a Brownian motion started from $z$ on the grid
	\[\mathcal{G}_m:=\left\{(x,y)\in\mathbb{R}^2\;\bigg|\;x\in\frac{1}{m}\mathbb{Z}^2\;\text{or}\;y\in\frac{1}{m}\mathbb{Z}^2\right\},\]
	as defined in Definition 4.3 of \cite{darrow2020convergence}. Define the stopping time
	\[\tau^*:=\inf\left\{s\;\bigg|\;W_{m,n}^{t}(s)\in\left(\frac{1}{m}\mathbb{Z}^2\setminus N_m(n)\right)\cup\partial D_\tau\right\},\]
	and set
	\[\nu_{\zeta,n}^{t+1}(z')=\nu_{\zeta,n}^{t}(z')+(\nu_{\zeta,n}^{t}(z)-1)\cdot\left(\mathbb{P}\left[W_{m,n}^{t}(\tau^*)=z\right]-\delta_{z,z'}\right).\]
	For a large enough $t'$, we have that $\nu_{\zeta,n}^{t'}(z)\leq1$ everywhere in $\supp\nu_{m,n}^{t'}\setminus\partial D_\tau$; then we define $\nu_{\zeta,n+1}=\nu_{\zeta,n}^{t'}$.
	
	In parallel with the original divisible sandpile model, we define $\nu_{\zeta,n+1}$ by taking the excess mass at $z$ in $\nu_{\zeta,n}$ and splitting it around the edge of $\supp\nu_{\zeta,n}$ according to a discrete harmonic measure. New in this case, however, is that we stop mass before it exits the domain $D_\tau$.
	
	Note that this satisfies the same key equality as the original harmonic measure; namely, for any grid harmonic (see \cite{darrow2020convergence}) $H$ defined in $\frac{1}{m}\mathbb{Z}^2\cap D_\tau$,
	\[\sum H\cdot\left(\nu_{\zeta,t}-\sigma_{\zeta,t}\right)=0.\]
\end{definition}

\begin{lemma}[Thin Tentacles]\label{tentaclessand}
	There is an absolute constant $b>0$ such that for all $z\in A_m(t)\subset\frac{1}{m}\mathbb{Z}^2$ with $d(z,D_0)\geq r$,
	\[\#\left(N_m(t)\cap B(z,r)\right)> bm^2r^2.\]
\end{lemma}
\begin{proof}
	For this, we define the intermediate processes $\nu_{m,t}^\ell$, for each integer $\ell\geq 1$:
	\begin{enumerate}
		\item Define the initial set $\nu^\ell_m(0)=\nu_m(0)=\mathbf{1}_{\frac{1}{m}\mathbb{Z}^2\cap D_0}$.
		\item For each $i\in\frac{1}{\ell}\mathbb{Z}_{>0}$, start a random walk at $z_{m,\lceil i\rceil}$, and let $z'_i$ be the first point in the walk at which $\nu^\ell_{m,i-\ell^{-1}}(z'_i)< 1$. Let $\nu^\ell_{m,i}=\nu^\ell_{m,i-\ell^{-1}}+\ell^{-1}\mathbf{1}_{\{z_{i}'\}}$.
	\end{enumerate}
	Now, $\nu_{m,t}^0$ is simply an IDLA, by definition, and $\nu_{m,t}^\ell\to\nu_{m,t}$ in law (pointwise) in $\ell$. We can now lift the proof of Lemma 2 of \cite{10.2307/23072157} (Lemma 3.2 of \cite{darrow2020convergence}) verbatim,\footnote{The only significant difference in proving the new result is that the total number of ``trials'', as well as the total number of required ``failures'' (in the language of \cite{10.2307/23072157}), is scaled up by a factor of $\ell$.} to show that
	\[\mathbb{P}\left[\nu^\ell_{m,t}(z)>0, \#\left(\{\nu^\ell_{m,t}=1\}\cap B(z,r)\right)\leq bm^2r^2\right]\leq C_0e^{-c_0mr}\]
	for constants $c_0,C_0$ independent of $\ell$. In particular, this probability is uniformly bounded below 1 in $\ell$; since $\nu_{m,t}^\ell$ converges in law to the deterministic function $\nu_{m,t}$, we see that 
	\[\lim_{\ell\to\infty}\mathbb{P}\left[\nu^\ell_{m,t}(z)>0, \#\left(\{\nu^\ell_{m,t}=1\}\cap B(z,r)\right)\leq bm^2r^2\right]=0,\]
	from which the lemma follows.
\end{proof}
The next theorem is a restatement of Theorem \ref{narrowsand}, and a stronger version of Theorem 3.9 of \cite{Levine_2008}:
\begin{theorem}[Theorem \ref{narrowsand}]
	There is a constant $C>0$ dependent on the flow such that, for large enough $m$ and any $s\in[0,T]$,
	\[\frac{1}{m}\mathbb{Z}^d\cap(D_s)_{Cm^{-3/5}}\subset N_m(m^2s)\subset (D_s)^{Cm^{-3/5}}.\]
\end{theorem}

\begin{lemma}\label{earlysand}
	There are constants $C,\alpha>0$ dependent only on the flow such that, for large enough $m$, $s\in[0,T]$, $a\geq Cm^{2/5}$, and $\ell\leq\alpha a$, an $a/m$-early point in $N_m(t)$ by time $m^2s$ implies a different, $\ell/m$-late point by time $m^2s$.
\end{lemma}
\begin{proof}
	Suppose $z\in N_m(t)\setminus N_m(t-1)$ is the first $a/m$-early point in $N_m$---that is, $z\notin (D_{t/m^2})^{a/m}$, but $N_m(t-1)\subset(D_{(t-1)/m^2})^{a/m}$. Further assume that there are no $\ell/m$-late points by the time $t$, or equivalently that $(D_{t/m^2})_{\ell/m}\subset N_m(t)$.
	
	Since $z$ is adjacent to $N_m(t-1)$, we know that \[N_m(t)\subset(D_{t/m^2})^{(a+1)/m}.\]
	Let $\zeta=\zeta(z,t)$ be the nearest point to $z$ in the annulus
	\[\frac{1}{m}\mathbb{Z}^2\cap(D_{t/m^2})^{V(4a+2C)/mv +2/m}\setminus(D_{t/m^2})^{V(4a+2C)/mv},\]
	where $C$ will be specified later, and $v,V>0$ are as in Lemma 3.5 of \cite{darrow2020convergence}. Let $\tau>0$ be such that $\zeta\in\partial D_\tau$, and note that
	\[d_H(D_{t/m^2},D_\tau)\geq d(D_{t/m^2},\zeta)\geq V(4a+2C)/mv.\]
	By Lemma 3.5 of \cite{darrow2020convergence}, this implies
	\begin{align*}
		d(N_m(t),D_\tau^c)&\geq d((D_{t/m^2})^{(a+1)/m},D_\tau^c)\\
		&\geq d(D_{t/m^2},D_\tau^c)-(a+1)/m\\
		&\geq \frac{v}{V}d_H(D_{t/m^2},D_\tau)-(a+1)/m\\
		&\geq (3a+2C-1)/m>1.
	\end{align*}
	From Lemma 4.2(a) of \cite{darrow2020convergence}, this means $\supp(\nu_{m,t})\subset \Omega_\zeta$, and thus we can replace $\nu_{m,t}$ with $\nu_{\zeta,t}$. As in \cite{darrow2020convergence}, we can show that---if $\#\left(N_m(t)\cap B(z,a/m)\right)>ba^2$ and no points are $\ell/m$-late by time $t$---then
	\[\sum_{z'\in\frac{1}{m}\mathbb{Z}^2}\left(\nu_{\zeta,t}(z')-\sigma_{m,t}(z')\right)H_\zeta(z')\geq vba/12V>0.\]
	Both of the listed assumptions are true; we know $\#\left(N_m(t)\cap B(z,a/m)\right)>ba^2$ by Lemma \ref{tentaclessand}, and we have assumed that no points are $\ell/m$-late by time $t$. However, $\sum\left(\nu_{\zeta,t}-\sigma_{m,t}\right)H=0$ for any $H$ harmonic on $\supp(\nu_{\zeta,t})$, so this is a contradiction.
\end{proof}

\begin{lemma}\label{latesand}
	There is a constant $C>0$ dependent only on the flow such that, for large enough $m$, $s\in[0,T]$, and $\ell\geq Cm^{2/5}$, there can be no $\ell/m$-late point by time $m^2s$. 
\end{lemma}
\begin{proof}
	Without loss of generality, let $a=\ell^2/Cm^{2/5}\geq\ell$. Fix an integer $t\leq m^2s$, and suppose that $\zeta\in\frac{1}{m}\mathbb{Z}^2\cap\left((D_{t/m^2})_{\ell/m}\setminus D_0\right)$ is $\ell/m$-late by time $t$. Then $d(\zeta,\partial D_{t/m^2})\geq\ell/m$, so by Lemma 3.5 of \cite{darrow2020convergence},
	\begin{align*}
		t-m^2\tau\geq 2m^2\sqrt{1+\tau}&\left(\sqrt{1+t/m^2}-\sqrt{1+\tau}\right)\\
		&\geq \frac{2m^2}{V}\sqrt{1+\tau}\cdot d_H(D_{t/m^2},D_\tau)\geq \frac{2m^2}{V}\sqrt{1+\tau}\cdot d(\zeta,\partial D_{t/m^2})\geq \frac{2m\ell}{V}.
	\end{align*}
	Since $\zeta$ is $\ell/m$-late at time $t$, we know that $\zeta\notin N_m(t)$, so $\nu_{\zeta,t}(\zeta)<1$. As in \cite{darrow2020convergence}, the sum
	\[\tilde{M}_\zeta(t):=\sum\nolimits_{z'}\left(\nu_{\zeta,t}(z')-\sigma_{m,t}(z')\right)H_\zeta(z')\]
	is maximized if the interior of $\tilde{\Omega}_\zeta\cap\frac{1}{m}\mathbb{Z}^2$ is fully occupied by $N_m(t)$. We can then show, exactly as in \cite{darrow2020convergence}, that
	\[\tilde{M}_\zeta(t)<1-c\ell\]
	for some $c>0$; the new constant term comes from the $\nu_{\zeta,t}(\zeta)H_\zeta(\zeta)<1$ contribution. So long as $C$ is large enough, this is still negative; of course, we know that $\tilde{M}_\zeta(t)=0$, so this is a contradiction.
\end{proof}

\begin{proof}[Proof of Theorem \ref{narrowsand}]
	We can work with the set $N_m(t)$ instead of $\supp\nu_{m,t}$; indeed, the latter only differs from the former within one unit of the boundary. By Lemma \ref{latesand}, we only need to show that no $O(m^{-3/5})$-early point can exist. Suppose a point is $\alpha^{-1}Cm^{-3/5}$-early at a time $t\leq m^2s$. By Lemma \ref{earlysand}, this implies that another point is $Cm^{-3/5}$-late by the same time; this contradicts Lemma \ref{latesand}, and we retrieve our result.
\end{proof}